\documentclass[final,3p,mathptmx]{elsarticle}

\usepackage{amsmath,amsthm,amssymb,amsfonts}
\usepackage{color}
\usepackage{rotating}

\newcommand{\supp}{\mathop{\mathrm{supp}}}

\theoremstyle{plain}
\newtheorem{thm}{Theorem}[section]
\newtheorem{lem}[thm]{Lemma}
\newtheorem{prop}[thm]{Proposition}
\newtheorem{cor}[thm]{Corollary}

\theoremstyle{definition}

\theoremstyle{remark}
\newtheorem{rem}[thm]{Remark}

\everymath{\displaystyle}
\begin{document}

\title{
Dual Univariate Interpolatory Subdivision of Every Arity:\\ Algebraic Characterization and Construction
}

\author[unibo]{Lucia Romani\corref{cor1}}
\ead{lucia.romani@unibo.it}
\cortext[cor1]{Corresponding author}
\author[unibo]{Alberto Viscardi}
\ead{alberto.viscardi@unibo.it}

\address[unibo]{Dipartimento di Matematica, Universit\`a di Bologna, Piazza di Porta San Donato 5, Bologna, Italy}

\date{}

\begin{abstract}
	
	A new class of univariate stationary interpolatory subdivision schemes of dual type is presented. As opposed to classical primal interpolatory schemes, these new schemes have masks with an even number of elements and are not step-wise interpolants. A complete algebraic characterization, which covers every arity, is given in terms of identities of trigonometric polynomials associated to the schemes. This characterization is based on a necessary condition for refinable functions to have prescribed values at the nodes of a uniform lattice, as a consequence of the Poisson summation formula. A strategy for the construction is then showed, alongside meaningful examples for applications that have comparable or even superior properties, in terms of regularity, length of the support and/or polynomial reproduction, with respect to the primal counterparts.

	\bigskip
	\noindent{\bf Classification (MSC2010): 65D05; 65D17; 41A05}


\end{abstract}

\maketitle

\section{Introduction}

Subdivision schemes are iterative methods for the construction of curves and surfaces exploited in various applications,
ranging from computer-aided geometric design, computer graphics and animation (see, e.g., \cite{MR2415757,Warren:2001:SMG:580358}) to the construction of wavelets and frames (see, e.g.,  \cite{MR1968118,MR1971300}).
Here we focus on univariate stationary subdivision \cite{MR1079033}. Convergent subdivision schemes are characterized by a compactly supported function $\varphi\in\mathcal{C}^0(\mathbb{R})$, called \emph{basic limit function}, which satisfies a  \emph{refinement equation} of the type
\begin{equation} \label{eq:ref_eq}
	\varphi(x)\;=\;\sum_{k\in\mathbb{Z}}\;a_k\;\varphi(mx-k+\tau),\quad x\in\mathbb{R},
\end{equation}
where $m\in\mathbb{N}\setminus\{1\}$ is the \emph{arity}, $\mathbf{a}=\{a_k\in\mathbb{R}\}_{k
\in\mathbb{Z}}$, is a compactly supported sequence of real values called \emph{mask} and $\tau\in[0,1)\cap\mathbb{Q}$ is a shift parameter. If a compactly supported function $\varphi$ satisfies \eqref{eq:ref_eq}, we can always consider it to have $\supp(\varphi)=[-s,s]$, $s>0$, changing the indexing and the shift parameter $\tau$ properly. In particular, denoting
\[
	k_\ell\;=\;\min\{\;k\in\mathbb{Z}\;:\;a_k\neq0\;\}\quad\textrm{ and }\quad k_r\;=\;\max\{\;k\in\mathbb{Z}\;:\;a_k\neq0\;\},
\]
it is easy to prove that $\varphi$ has a symmetric support if and only if one of the following is true:
\begin{equation} \label{eq:pd}
	\begin{array}{l}
		(p)\qquad \tau=0 \quad\textrm{ and }\quad k_\ell=-k_r; \\ \\
		(d)\qquad \tau=\frac{1}{2} \quad\textrm{ and }\quad k_\ell=1-k_r.
	\end{array}	
\end{equation}
Because of \eqref{eq:pd}, univariate subdivision schemes are divided in two major families: primal schemes, satisfying $(p)$, and dual schemes, satisfying $(d)$.\\
By constructing the subdivision symbol
\begin{equation} \label{eq:symbol}
	A(z)\;=\;\frac{1}{m}\;\sum_{k=k_\ell}^{k_r}\;a_k\;z^k,\quad z\in\mathbb{C},\;|z|=1,
\end{equation}
associated to the mask $\mathbf{a}=\{a_k\in\mathbb{R}\}_{k=k_\ell,\ldots,k_r}$, it is well-known (see, e.g., \cite{MR2775138,MR2843037,MR3071114}) that the shift parameter $\tau$ satisfies
$$\tau\;=\;A'(1).$$
The target of this work are interpolatory schemes, i.e., schemes with a basic limit function $\varphi$ satisfying
\begin{equation} \label{eq:interp}
	\varphi(n)\;=\;\delta_{0,n},\quad n\in\mathbb{Z}.
\end{equation}
Within the class of interpolatory schemes, primal ones are characterized by a simple polynomial equation which involves the \emph{sub-symbols} of the scheme, i.e., the Laurent polynomials
\begin{equation} \label{eq:sub_sym_A}
	A_{n}(z)\;=\;\frac{1}{m}\;\sum_{k\in\mathbb{Z}}\;a_{mk+n}\;z^{mk+n}, \quad n=0, \ldots, m-1 \quad \hbox{satisfying} \quad A(z)\;=\;\sum_{n=0}^{m-1}\;A_n(z).
\end{equation}
In particular, a convergent subdivision scheme is primal interpolatory if and only if \cite{MR1397613}
\begin{equation} \label{eq:primal_interp}
	A_{0}(z)\;\equiv\;\frac{1}{m},\quad\forall z\in\mathbb{C},\;|z|=1.
\end{equation}
Due to \eqref{eq:primal_interp}, primal interpolatory schemes have the so called \emph{stepwise interpolation property}, i.e., at each subdivision step they maintain the data of the previous one.\\
Beside the use of primal interpolatory schemes, there exist other approaches to interpolate points via subdivision, such as those that apply an approximating scheme after suitably preprocessing the data to be interpolated (see, e.g., \cite{Deng2010137,MR3921224,Zheng2006301}).
However, before \cite{LUCIA}, none of the existing approaches took into account the possibility of constructing a native dual interpolatory scheme which does not have the property of retaining the initial data at each iteration,
but achieves the interpolation in the sense that the initial data are still preserved in the limit function.
All dual subdivision schemes we can find in literature are indeed not interpolatory \cite{MR3071114},
except for the family of dual quaternary schemes introduced in \cite{LUCIA}
and for the class of dual $(2n)$-point subdivision schemes proposed in \cite{DENG2019344}, which is shown to possess the interpolation property only when $n$ tends to infinity, i.e., when the subdivision mask has infinite length.\\
To fill this theoretical gap in the literature, in this paper we investigate dual interpolatory schemes with finite masks  and present a complete algebraic characterization of their symbols, which covers every arity.
Additionally, for any arbitrary arity $m$ greater than $2$, this characterization is used as a general constructive method
to produce a great amount of new interpolatory schemes with different features.
In fact, the method we propose allows the user to tune a good amount of degrees of freedom: the arity $m$, the length of the mask $\mathbf{a}$, the desired degree of polynomial reproduction and some samples of the resulting basic limit function $\varphi$.\\

The remainder of the paper is organized as follows. In Section \ref{sec:nec_cond} we state a necessary condition for a basic limit function to have prescribed values at lattices of the form \(\mathbb{Z}/T\), for $\tau\in[0,1)\cap\mathbb{Q}$ and \(T\in\mathbb{N}\setminus\{0\}\) satisfying \(\tau T\in\mathbb{N}\). In Section \ref{sec:interp_char}, we show an explicit algebraic characterization of dual interpolatory schemes in terms of polynomial equalities involving sub-symbols and other polynomials related to the evaluation of the basic limit function at $\mathbb{Z}/2$. Exploiting this fact we are able to construct, in Section \ref{sec:dual_interp}, new and interesting interpolatory schemes never considered in the literature, which possess
comparable or even superior properties, in terms of
regularity, length of the support and/or polynomial reproduction, with respect to their primal counterparts.
Conclusions are drawn in Section \ref{sec:concl}.

\section{Necessary Algebraic Condition for Refinability} \label{sec:nec_cond}

We start proving a general necessary condition for a function $\varphi$ to be the basic limit function of a convergent subdivision scheme, having prescribed values over the lattice $\mathbb{Z}/T$, $T>0$. The proof exploits the notation
\begin{equation} \label{eq:Fourier_transform}
	\widehat{f}(\omega)\;=\;\int_\mathbb{R}\;f(x)\;e^{-2\pi i x \omega}\;dx,\quad \omega\in\mathbb{R}
\end{equation}
to refer to the Fourier transform of a function $f\in L^1(\mathbb{R})$, and
is based on a fundamental result of harmonic analysis, i.e. the \emph{Poisson summation formula} in its generalized form \cite{MR1420504}, which states:\\
if a function $f:\mathbb{R}\rightarrow\mathbb{R}$, $f\in L^1(\mathbb{R})$  satisfies
\begin{equation} \label{eq:PSF_hyp}
	|f(x)|\;+\;|\widehat{f}(x)|\;\leq\;C(1+|x|)^{-1-\epsilon},\quad x\in\mathbb{R},
\end{equation}
for some $C,\epsilon>0$, then,
\begin{equation} \label{eq:PSF}
	\sum_{n\in\mathbb{Z}}\;\widehat{f}(\omega+nT)\;=\;\frac{1}{T}\;\sum_{n\in\mathbb{Z}}\;f\left(\frac{n}{T}\right)\;e^{-2\pi i \omega \frac{n}{T}},\quad \omega\in\mathbb{R},\;T>0.
\end{equation}

\begin{rem}
	Condition \eqref{eq:PSF_hyp} is easily satisfied by continuous compactly supported functions such as the basic limit functions of convergent subdivision schemes.
\end{rem}

\begin{thm} \label{thm:alg_char}
	Let \(\varphi\) be the basic limit function of a convergent subdivision scheme of arity $m\in\mathbb{N}\setminus\{0,1\}$, compactly supported mask $\mathbf{a}=\{a_\ell\}_{\ell\in\mathbb{Z}}$ and sub-symbols $A_k(z)$, $k\in\mathbb{Z}$, satisfying \eqref{eq:ref_eq} with $\tau\in[0,1)\cap\mathbb{Q}$. Then, for every $T\in\mathbb{N}\setminus\{0\}$ such that \(\tau T\in\mathbb{N}\),
	the following polynomial identity holds
	\begin{equation} \label{eq:sym_char}
		\sum_{\gamma=0}^{mT-1}\;\Phi_{T,\gamma}(z^m)\;=\;m\;z^{-\tau T}\;\sum_{\beta=0}^{m-1}\;\sum_{\substack{\gamma=0\\\gamma+\beta T\equiv_m \tau T}}^{mT-1}\;A_{\beta}(z^T)\;\Phi_{T,\gamma}(z),
	\end{equation}
	where
	\begin{equation} \label{eq:sub_sym_P}
		\Phi_{T,n}(z)\;=\;\frac{1}{T}\;\sum_{k\in\mathbb{Z}}\;\varphi\left(mk+\frac{n}{T}\right)\;z^{mTk+n},\quad n\in\mathbb{Z},
	\end{equation}
	with \(z\;=\;e^{-2\pi i \frac{\omega}{T}}\).
\end{thm}

\begin{proof}
	It is well known that, on the Fourier side, the refinement equation \eqref{eq:ref_eq} reads as
	\begin{equation} \label{eq:ref_eq_Fourier}
		\widehat{\varphi}(\omega)\;=\;e^{2\pi i \frac{\omega}{m} \tau}a\left(\frac{\omega}{m}\right)\;\widehat{\varphi}\left(\frac{\omega}{m}\right),
	\end{equation}
	where
	\begin{equation} \label{eq:sym}
		a(\omega)\;=\;\frac{1}{m}\;\sum_{k\in\mathbb{Z}}\;a_k\;e^{-2\pi i \omega k},
	\end{equation}
	is the \emph{symbol} associated to the refinement equation \eqref{eq:ref_eq}.
	Then, for every \(n\in\mathbb{Z}\), we get
	\[
		\widehat{\varphi}(m\omega+nT)\;=\;e^{2\pi i \left(\omega+\frac{nT}{m}\right)\tau}\;a\left(\omega+\frac{nT}{m}\right)\;\widehat{\varphi}\left(\omega+\frac{nT}{m}\right).
	\]
	Summing over $n\in\mathbb{Z}$, we obtain
	\[
		\begin{array}{rcl}
			\sum_{n\in\mathbb{Z}}\;\widehat{\varphi}(m\omega+nT)&=&\sum_{n\in\mathbb{Z}}\;e^{2\pi i \left(\omega+\frac{nT}{m}\right)\tau}\;a\left(\omega+\frac{nT}{m}\right)\;\widehat{\varphi}\left(\omega+\frac{nT}{m}\right) \\ \\
			&=& \sum_{\alpha=0}^{m-1}\;\sum_{h\in\mathbb{Z}}\;e^{2\pi i \left(\omega+hT+\frac{\alpha T}{m}\right)\tau}\;a\left(\omega+h T+\frac{\alpha T}{m}\right)\;\widehat{\varphi}\left(\omega+h T+\frac{\alpha T}{m}\right),
		\end{array}
	\]
	where we set $n=mh+\alpha$. Since $\tau T\in\mathbb{Z}$, we have
	\[
		\begin{array}{rcl}
			e^{2\pi i \left(\omega+hT+\frac{\alpha T}{m}\right)\tau}\;a\left(\omega+h T+\frac{\alpha T}{m}\right)&=&\frac{e^{2\pi i \left(\omega+\frac{\alpha T}{m}\right)\tau}}{m}\;\sum_{k\in\mathbb{Z}}\;a_k\;e^{-2\pi i \left(\omega+h T+\frac{\alpha T}{m}\right)k} \\ \\
			&=& \frac{z^{-\tau T}\;e^{2\pi i \frac{\alpha}{m} \tau T}}{m}\;\sum_{k\in\mathbb{Z}}\;e^{-2\pi i \frac{\alpha}{m}kT}\;a_k\;\left(e^{-2\pi i \frac{\omega}{T}}\right)^{kT}\\ \\
			&\underset{k=mn+\beta}{=}& \frac{z^{-\tau T}\;e^{2\pi i \frac{\alpha}{m} \tau T}}{m}\;\sum_{\beta=0}^{m-1}\;e^{-2\pi i \frac{\alpha}{m}\beta T}\;\sum_{n\in\mathbb{Z}}\;a_{mn+\beta}\;\left(e^{-2\pi i \frac{\omega}{T}}\right)^{(mn+\beta)T}\\ \\
			&=& z^{-\tau T}\;e^{2\pi i \frac{\alpha}{m} \tau T}\;\sum_{\beta=0}^{m-1}\;e^{-2\pi i \frac{\alpha}{m}\beta T}\;A_{\beta}(z^T).
		\end{array}
	\]
	Thus,
	\begin{equation} \label{eq:step1}
		\sum_{n\in\mathbb{Z}}\;\widehat{\varphi}(m\omega+nT)\;=\; z^{-\tau T}\;\sum_{\alpha=0}^{m-1}\;e^{2\pi i \frac{\alpha}{m} \tau T}\;\left(\;\sum_{\beta=0}^{m-1}\;e^{-2\pi i \frac{\alpha}{m}\beta T}\;A_{\beta}(z^T)\;\right)\;\sum_{h\in\mathbb{Z}}\;\widehat{\varphi}\left(\omega+h T+\frac{\alpha T}{m}\right).
	\end{equation}
	At this point we are ready to apply Poisson summation formula to both sides of \eqref{eq:step1} obtaining, for the left-hand-side,
	\begin{equation} \label{eq:step1a}
		\begin{array}{rcl}
			\sum_{n\in\mathbb{Z}}\;\widehat{\varphi}(m\omega+nT)&=&\frac{1}{T}\;\sum_{n\in\mathbb{Z}}\;\varphi\left(\frac{n}{T}\right)\;\left(e^{-2\pi i \frac{\omega}{T}m}\right)^n\\ \\
			&\underset{n=mTk+\gamma}{=}& \sum_{\gamma=0}^{mT-1}\;\frac{1}{T}\;\sum_{k\in\mathbb{Z}}\;\varphi\left(mk+\frac{\gamma}{T}\right)\;(z^m)^{mTk+\gamma}\\ \\
			&=&\sum_{\gamma=0}^{mT-1}\;\Phi_{T,\gamma}(z^m)
		\end{array}
	\end{equation}
	and, for the right-hand-side,
	\begin{equation} \label{eq:step1b}
		\begin{array}{rcl}
			\sum_{h\in\mathbb{Z}}\;\widehat{\varphi}\left(\omega+h T+\frac{\alpha T}{m}\right)&=&\frac{1}{T}\;\sum_{h\in\mathbb{Z}}\;\varphi\left(\frac{h}{T}\right)\;e^{-2 \pi i \left(\omega+\frac{\alpha T}{m}\right)\frac{h}{T}} \\ \\
			&\underset{h=mTn+\gamma}{=}& \sum_{\gamma=0}^{mT-1}\;\frac{1}{T}\;\sum_{n\in\mathbb{Z}}\;\varphi\left(mn+\frac{\gamma}{T}\right)\;e^{-2 \pi i \left(\omega+\frac{\alpha T}{m}\right)\frac{mTn+\gamma}{T}} \\ \\
			&=&
			\sum_{\gamma=0}^{mT-1}\;e^{-2\pi i \frac{\alpha}{m} \gamma}\;\Phi_{T,\gamma}(z).
		\end{array}
	\end{equation}
	Combining \eqref{eq:step1} with \eqref{eq:step1a} and \eqref{eq:step1b} we obtain
	\[
		\begin{array}{rcl}
			\sum_{\gamma=0}^{mT-1}\;\Phi_{T,\gamma}(z^m)&=&z^{-\tau T}\;\sum_{\alpha=0}^{m-1}\;e^{2\pi i \frac{\alpha}{m} \tau T}\;\left(\;\sum_{\beta=0}^{m-1}\;e^{-2\pi i \frac{\alpha}{m}\beta T}\;A_{\beta}(z^T)\;\right)\;\left(\;\sum_{\gamma=0}^{mT-1}\;e^{-2\pi i \frac{\alpha}{m} \gamma}\;\Phi_{T,\gamma}(z)\;\right) \\ \\
			&=&z^{-\tau T}\;\sum_{\beta=0}^{m-1}\;\sum_{\gamma=0}^{mT-1}\;A_{\beta}(z^T)\;\Phi_{T,\gamma}(z)\;\sum_{\alpha=0}^{m-1}\;e^{-2\pi i \frac{\alpha}{m}(\gamma+\beta T - \tau T)}.
		\end{array}	
	\]
	Since
	\[
		\sum_{\alpha=0}^{m-1}\;e^{-2\pi i \frac{\alpha}{m}(\gamma+\beta T - \tau T)}\;=\;
		\left\{\begin{array}{cl}
			m,&\textrm{ if } \gamma+\beta T\equiv_m\tau T,\\ \\
			0,&\textrm{ otherwise,}
		\end{array}\right.
	\]
we finally arrive at \eqref{eq:sym_char}.
\end{proof}

\begin{rem}
	For \(m=2,\;\tau=0,\;T=1\), the identity \eqref{eq:sym_char}
	is equivalent to the algebraic condition stated in \cite{MR1790328}, equation $(2.7)$, for a binary refinable function to have prescribed values at the integers,
	thus it can be seen as a generalization of it.
\end{rem}

\begin{rem} \label{rem:min_T}
	In Theorem \ref{thm:alg_char}, once \(m\) and \(\tau\) are fixed, the bigger is \(T\), the more evaluations of \(\varphi\) we need to take into account. In particular, since the aim is to construct a refinable function given some of its values, it makes sense to consider
	\begin{equation} \label{eq:min_T}
		T\;=\;\min\{\;k\in\mathbb{N}\setminus\{0\}\;:\;\tau k\in\mathbb{N}\;\}.
	\end{equation}
	Indeed, if there are no trigonometric polynomials \(\{A_{\beta}\}_{\beta=0}^{m-1}\) satisfying \eqref{eq:sym_char} with \(T\) as in \eqref{eq:min_T} for a fixed set of values \(\{\varphi(k/T)\}_{k\in\mathbb{Z}}\), there is no hope about the existence of solutions considering \(kT\), \(k\in\mathbb{N}\setminus\{0,1\}\), instead of \(T\), since \(\mathbb{Z}/T\subset\mathbb{Z}/(kT)\).
\end{rem}

\section{Algebraic Characterization} \label{sec:interp_char}

Let us consider the dual case, i.e., $(d)$ in \eqref{eq:pd}. Since $\tau=\frac{1}{2}$, according to \eqref{eq:min_T}  we choose $T=2$. In Lemma \ref{lem:nec_cond} we start specializing the necessary condition of Theorem \ref{thm:alg_char} to the dual interpolatory case. Then we split the characterization based on the arity $m$ being odd (Theorem \ref{thm:dual_char_odd}) or even (Theorem \ref{thm:dual_char_even}).

\begin{lem} \label{lem:nec_cond}
	Consider a convergent $m$-ary dual interpolatory subdivision scheme with compactly supported mask $\mathbf{a}=\{a_\ell\}_{\ell\in\mathbb{Z}}$, sub-symbols $A_k(z)$, $k\in\mathbb{Z}$, and basic limit function $\varphi\in\mathcal{C}^0(\mathbb{R})$ with compact symmetric support, that satisfies $\varphi(k)=\delta_{0,k}$, $k\in\mathbb{Z}$. Then
	\begin{equation} \label{eq:lemma}
		\frac{1}{2}\;+\;\sum_{\gamma=0}^{m-1}\;\Phi_{2,2\gamma+1}(z^m)\;=\;m\;z^{-1}\;\left(\;\sum_{\substack{\beta=0\\2\beta\equiv_m 1}}^{m-1}\;\frac{A_\beta(z^2)}{2}\;+\;\sum_{\substack{\beta,\gamma=0\\2(\gamma+\beta) \equiv_m 0}}^{m-1}\;A_\beta(z^2)\;\Phi_{2,2\gamma+1}(z)\;\right).
	\end{equation}
\end{lem}

\begin{proof}
	First we rewrite the necessary condition for refinability \eqref{eq:sym_char} with $\tau=\frac{1}{2}$ and $T=2$, obtaining
	\begin{equation} \label{eq:step2}
		\sum_{\gamma=0}^{2m-1}\;\Phi_{2,\gamma}(z^m)\;=\;m\;z^{-1}\;\sum_{\beta=0}^{m-1}\;\sum_{\substack{\gamma=0\\\gamma+2\beta \equiv_m 1}}^{2m-1}\;A_{\beta}(z^2)\;\Phi_{2,\gamma}(z).
	\end{equation}
	The fact that $\varphi(k)=\delta_{0,k}$, $k\in\mathbb{Z}$, together with the definition \eqref{eq:sub_sym_P} implies that,
	\[
		\Phi_{2,2\gamma}(z)\;=\;\left\{\begin{array}{cl}
			\frac{1}{2}, &\textrm{ if } \gamma=0, \\ \\
			0, &\textrm{ otherwise.}
		\end{array}\right.
	\]
	Thus, dividing the sums over $\gamma$ in \eqref{eq:step2} into sums given by the indices $2\gamma+j$, $\gamma=0,\dots,m-1$, $j=0,1$, we obtain
	\[
		\sum_{\gamma=0}^{2m-1}\;\Phi_{2,\gamma}(z^m)\;=\;\frac{1}{2}\;+\;\sum_{\gamma=0}^{m-1}\;\Phi_{2,2\gamma+1}(z),
	\]
	and
	\[
		m\;z^{-1}\;\sum_{\beta=0}^{m-1}\;\sum_{\substack{\gamma=0\\\gamma+2\beta \equiv_m 1}}^{2m-1}\;A_{\beta}(z^2)\;\Phi_{2,\gamma}(z) \;=\; m\;z^{-1}\;\left(\;\sum_{\substack{\beta=0\\2\beta\equiv_m 1}}^{m-1}\;\frac{A_\beta(z^2)}{2}\;+\sum_{\substack{\beta,\gamma=0\\2(\gamma+\beta) \equiv_m 0}}^{m-1}\;A_{\beta}(z^2)\;\Phi_{2,2\gamma+1}(z)\;\right).
	\]
This completes the proof.
\end{proof}

As a consequence of Lemma \ref{lem:nec_cond}, it is easy to prove that for the arity $m=2$ there are no continuous refinable functions. This could be the reason why dual interpolatory schemes have not been investigated before.

\begin{cor}
	For $m=2$ there are no convergent dual interpolatory subdivision schemes.
\end{cor}

\begin{proof}
	Consider the necessary condition \eqref{eq:lemma} for $m=2$:
	\[
		\frac{1}{2}\;+\;\sum_{\gamma=0}^{1}\;\Phi_{2,2\gamma+1}(z^2)\;=\;2\;z^{-1}\;\left(\;\sum_{\substack{\beta=0\\2\beta\equiv_2 1}}^{1}\;\frac{A_\beta(z^2)}{2}\;+\;\sum_{\substack{\beta,\gamma=0\\2(\gamma+\beta) \equiv_2 0}}^{1}\;A_\beta(z^2)\;\Phi_{2,2\gamma+1}(z)\;\right).
	\]
	In particular, the left-hand side can be rewritten as
	\[
		\frac{1}{2}\;+\;\sum_{\gamma=0}^{1}\;\Phi_{2,2\gamma+1}(z^2)\;=\;\frac{1}{2}\;+\;\Phi_{2,1}(z^2)\;+\;\Phi_{2,3}(z^2)\;=:\;F(z^2).
	\]
	Then, recalling \eqref{eq:sub_sym_A}, we rewrite the right-hand side as
	\[\begin{array}{l}
		2\;z^{-1}\;\left(\;\sum_{\substack{\beta=0\\2\beta\equiv_2 1}}^{1}\;\frac{A_\beta(z^2)}{2}\;+\;\sum_{\substack{\beta,\gamma=0\\2(\gamma+\beta) \equiv_2 0}}^{1}\;A_\beta(z^2)\;\Phi_{2,2\gamma+1}(z)\;\right)\;= \\ \\ \qquad\qquad\qquad\qquad\qquad\qquad\qquad\qquad=\;2\;z^{-1}\;\left(\;A_0(z^2)\;+\;A_1(z^2)\;\right)\;\left(\;\Phi_{2,1}(z)\;+\;\Phi_{2,3}(z)\;\right) \\ \\
		\qquad\qquad\qquad\qquad\qquad\qquad\qquad\qquad=\;\sum_{k\in\mathbb{Z}}\;a_k\;z^{2k-1}\;\left(\;F(z)\;-\;\frac{1}{2}\;\right).
	\end{array}\]
	Thus, since
	\[
		F(z)\;=\;\sum_{j\in\mathbb{Z}}\;f_j\;z^j\quad\textrm{ with }\quad f_j\;=\;\left\{\begin{array}{cl}
				\frac{1}{2},&j=0,\\ \\
				\varphi\left(2k+\frac{1}{2}\right),&j=4k+1,\\ \\
				\varphi\left(2k+\frac{3}{2}\right),&j=4k+3,\\ \\
				0,&\textrm{ otherwise,}
			\end{array}\right.
	\]
	we get
	\begin{equation} \label{eq:bin_step}
		\sum_{j\in\mathbb{Z}}\;f_j\;z^{2j}\;=\;\sum_{k\in\mathbb{Z}}\;a_k\;z^{2k-1}\;\sum_{j\in\mathbb{Z}\setminus\{0\}}\;f_j\;z^j.
	\end{equation}
	Now both $\mathbf{a}=\{a_k\}_{k\in\mathbb{Z}}$ and $\varphi$ (and so $\mathbf{f}=\{f_j\}_{j\in\mathbb{Z}}$) are compactly supported. Then, because of \eqref{eq:bin_step}, the first and the last non-zero elements of $\mathbf{a}$ must be equal to $1$. There follows that the associated subdivision scheme can not be convergent, since its difference scheme cannot be contractive \cite{CHOI2006351,MR1172120,MR2008967}.
\end{proof}

\begin{rem}
	Actually a dual interpolatory scheme was almost known since 1884. The scheme is the one with arity $m=3$ and mask $\mathbf{a}= \left \{\; \frac{1}{2},\;1,\;1,\; \frac{1}{2}\; \right \}$, whose basic limit function is related to the \emph{Cantor function} (see, e.g., \cite{MR2195181}). Indeed, the resulting basic limit function $\varphi$, Figure \ref{fig:Cantor}, has $\supp(\varphi)= \left [-\frac{3}{4}, \frac{3}{4} \right ]$ and is divided into three parts: over $\left [-\frac{3}{4}, -\frac{1}{4} \right ]$ it is exactly the well-known ascending Cantor function, over $\left [ -\frac{1}{4},\frac{1}{4} \right ]$ it is constant equal to $1$ and over $\left [ \frac{1}{4},\frac{3}{4} \right ]$ it is equal to the descending Cantor function. This scheme however reproduces only constants and $\varphi\in\mathcal{C}^{\log_3(2)}(\mathbb{R})$. So it is not really useful for applications. This is the shortest possible basic limit function obtained by a converging dual interpolatory scheme.
\end{rem}

\begin{figure}[h!]
	\centering
	\includegraphics[scale=0.45]{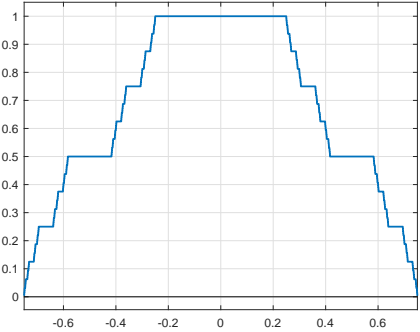}
	\caption{The basic limit function of the ternary scheme related to the Cantor function, which is supported in $[-0.75,0.75]$, reproduces constants and belongs to $\mathcal{C}^{\log_3(2)}(\mathbb{R})$.}
	\label{fig:Cantor}
\end{figure}

At this point, to proceed with the algebraic characterization of dual interpolatory schemes we have to split computations into two cases: the case with $m$ odd (Theorem \ref{thm:dual_char_odd}) and the case with $m$ even (Theorem \ref{thm:dual_char_even}). The changes are due to the equivalences that the indices of the sums in the right-hand-side of \eqref{eq:lemma} must satisfy, which depend on $m$. We thus treat the two cases separately.

\begin{thm} \label{thm:dual_char_odd}
	Let $m=2\lambda+1,\;\lambda\in\mathbb{N}\setminus\{0\}$. A convergent $m$-ary subdivision scheme is a dual interpolatory scheme if and only if
	\begin{equation} \label{eq:dual_char_odd} \frac{1}{2}\;+\;\sum_{\gamma=0}^{m-1}\;\Phi_{2,2\gamma+1}(z^{m})\;=\;m\;z^{-1}\;\left(\;\frac{A_{\frac{m+1}{2}}(z^2)}{2}\;+\;\sum_{\gamma=0}^{m-1} \;A_{m-\gamma}(z^2)\;\Phi_{2,2\gamma+1}(z)\;\right).
	\end{equation}
\end{thm}

\begin{proof}
	We start from \eqref{eq:lemma} and we need to further study the right-hand-side. We observe that the only $\beta\in\{0,\dots,m-1\}$ satisfying $2\beta\equiv_m 1$ is $\beta=\frac{m+1}{2}$. On the other hand, the only possibility to satisfy $2(\beta+\gamma)\equiv_m 0$ for $\beta,\gamma\in\{0,\dots,m-1\}$ is that $\beta+\gamma\in\{0,m\}$, i.e., $\beta=m-\gamma$ if $\beta,\gamma\neq0$, and $\beta=\gamma=0$. Thus, the expressions in the right-hand side of \eqref{eq:lemma} can be rewritten as
	\[\begin{array}{l}
		\sum_{\substack{\beta=0\\2\beta\equiv_m 1}}^{m-1}\;\frac{A_\beta(z^2)}{2}\;+\;\sum_{\substack{\beta,\gamma=0\\2(\gamma+\beta) \equiv_m 0}}^{m-1}\;A_\beta(z^2)\;\Phi_{2,2\gamma+1}(z) \;=\\ \\
		\qquad\qquad\qquad\qquad\qquad=\;\frac{A_{\frac{m+1}{2}}(z^2)}{2}\;+\;A_{0}(z^2)\;\Phi_{2,1}(z)\;+\;\sum_{\gamma=1}^{m-1} \;A_{m-\gamma}(z^2)\;\Phi_{2,2\gamma+1}(z) \\ \\
		\qquad\qquad\qquad\qquad\qquad=\;\frac{A_{\frac{m+1}{2}}(z^2)}{2}\;+\;\sum_{\gamma=0}^{m-1} \;A_{m-\gamma}(z^2)\;\Phi_{2,2\gamma+1}(z),
	\end{array}\]
	where the last equality holds since $A_{0}(z)\;=\;A_{m}(z)$. Thus, equation \eqref{eq:dual_char_odd} follows from \eqref{eq:lemma}.
	
	Conversely, if \eqref{eq:dual_char_odd} holds, we can subtract it from the necessary condition for refinability \eqref{eq:sym_char} with $\tau=\frac{1}{2}$ and $T=2$, getting, for the left-hand-side,
	\begin{equation} \label{eq:LHS_odd}
		\sum_{\gamma=0}^{2m-1}\;\Phi_{2,\gamma}(z^m)-\frac{1}{2}\;-\;\sum_{\gamma=0}^{m-1}\;\Phi_{2,2\gamma+1}(z^{m})\;=\;\sum_{\gamma=0}^{m-1}\;\Phi_{2,2\gamma}(z^m)\;-\;\frac{1}{2},
	\end{equation}
	and, for the right-hand-side,
	\begin{equation} \label{eq:RHS_odd}\begin{array}{l}
		m\;z^{-1}\;\sum_{\beta=0}^{m-1}\;\sum_{\substack{\gamma=0\\\gamma+2\beta \equiv_m 1}}^{2m-1}\;A_{\beta}(z^2)\;\Phi_{2,\gamma}(z)\;-\;m\;z^{-1}\;\left(\;\frac{A_{\frac{m+1}{2}}(z^2)}{2}\;+\;\sum_{\gamma=0}^{m-1} \;A_{m-\gamma}(z^2)\;\Phi_{2,2\gamma+1}(z)\;\right)\;=\\ \\
		\qquad\qquad=\;m\;z^{-1}\;\Bigg(\;\sum_{\substack{\beta,\gamma=0\\2(\gamma+\beta) \equiv_m 1}}^{m-1}\;A_{\beta}(z^2)\;\Phi_{2,2\gamma}(z)\;+\;\sum_{\substack{\beta,\gamma=0\\2(\gamma+\beta) \equiv_m 0}}^{m-1}\;A_{\beta}(z^2)\;\Phi_{2,2\gamma+1}(z)\;-\;\frac{A_{\frac{m+1}{2}}(z^2)}{2} \\ \\
		\hfill -\;\sum_{\gamma=0}^{m-1} \;A_{m-\gamma}(z^2)\;\Phi_{2,2\gamma+1}(z)\;\Bigg) \\ \\
		\qquad\qquad=\;m\;z^{-1}\;\left(\;\sum_{\gamma=0}^{\frac{m+1}{2}}\;A_{\frac{m+1}{2}-\gamma}(z^2)\;\Phi_{2,2\gamma}(z)\;+\;\sum_{\gamma=\frac{m+3}{2}}^{m-1}\;A_{\frac{3m+1}{2}-\gamma}(z^2)\;\Phi_{2,2\gamma}(z)\;-\;\frac{A_{\frac{m+1}{2}}(z^2)}{2}\right.\\ \\
		\hfill\left.+\;A_{0}(z^2)\;\Phi_{2,1}(z)\;+\;\sum_{\gamma=1}^{m-1}\;A_{m-\gamma}(z^2)\;\Phi_{2,2\gamma+1}(z)\;-\;\sum_{\gamma=0}^{m-1} \;A_{m-\gamma}(z^2)\;\Phi_{2,2\gamma+1}(z)\;\right) \\ \\
		\qquad\qquad=\;m\;z^{-1}\left(\;\sum_{\gamma=0}^{m-1}\;A_{\frac{m+1}{2}-\gamma}(z^2)\;\Phi_{2,2\gamma}(z)\;-\;\frac{A_{\frac{m+1}{2}}(z^2)}{2}\;\right),
	\end{array}\end{equation}
	where we used the fact that $A_{\frac{3m+1}{2}-\gamma}(z)=A_{\frac{m+1}{2}-\gamma}(z)$. Combining \eqref{eq:LHS_odd} and \eqref{eq:RHS_odd}, we obtain
	\begin{equation} \label{eq:step}
		\begin{array}{rcl}
			\frac{1}{2}&=&\sum_{\gamma=0}^{m-1}\;\Phi_{2,2\gamma}(z^m)\;+\;m\;z^{-1}\left(\;\frac{A_{\frac{m+1}{2}}(z^2)}{2}\;-\;\sum_{\gamma=0}^{m-1}\;A_{\frac{m+1}{2}-\gamma}(z^2)\;\Phi_{2,2\gamma}(z)\;\right).
		\end{array}
	\end{equation}
	Recalling \eqref{eq:sub_sym_A} and \eqref{eq:sub_sym_P}, one realizes that all the powers of $z$ of the second term of the right-hand-side of \eqref{eq:step} are odd, and thus, since the left-hand side $\frac{1}{2}$ presents a unique power of $z$ which is even, it must hold that
	\begin{equation} \label{eq:step3} \frac{1}{2}\;=\;\sum_{\gamma=0}^{m-1}\;\Phi_{2,2\gamma}(z^m)\;=\;\frac{1}{2}\;\sum_{\gamma=0}^{m-1}\;\sum_{k\in\mathbb{Z}}\;\varphi\left(mk+\gamma\right)z^{2m(mk+\gamma)},
	\end{equation}
	which implies $\varphi(k)\;=\;\delta_{0,k}$, $k\in\mathbb{Z}$, and this concludes the proof.
\end{proof}

\begin{thm} \label{thm:dual_char_even}
	Let $m= 2\lambda,\;\lambda\in\mathbb{N}\setminus\{0,1\}$. A convergent $m$-ary subdivision scheme is a dual interpolatory scheme if and only if
	\begin{equation} \label{eq:dual_char_even}
		\frac{1}{2}\;+\;\sum_{\gamma=0}^{m-1}\;\Phi_{2,2\gamma+1}(z^{m})\;=\;m\;z^{-1}\;\sum_{\gamma=0}^{m-1} \;\bigg(\;A_{\frac{m}{2}-\gamma}(z^2)\;+\;A_{m-\gamma}(z^2)\;\bigg)\;\Phi_{2,2\gamma+1}(z).
	\end{equation}
\end{thm}

\begin{proof}
	As for the proof of Theorem \ref{thm:dual_char_odd}, we start from \eqref{eq:lemma} and we need to further study the right-hand-side. We first observe that no $\beta\in\{0,\dots,m-1\}$ satisfies $2\beta\equiv_m 1$, being $m$ even. On the other hand, the only possibility to satisfy $2(\beta+\gamma)\equiv_m 0$ for $\beta,\gamma\in\{0,\dots,m-1\}$ is that $\beta+\gamma\in \left \{ 0, \frac{m}{2}, m, \frac{3m}{2} \right \}$, namely $\beta=\gamma=0$ or $\beta= \frac{m}{2} -\gamma$, $\beta=m-\gamma$, $\beta= \frac{3m}{2} -\gamma$. Thus,
	\[\begin{array}{l}
		\sum_{\substack{\beta=0\\2\beta\equiv_m 1}}^{m-1}\;\frac{A_\beta(z^2)}{2}\;+\;\sum_{\substack{\beta,\gamma=0\\2(\gamma+\beta) \equiv_m 0}}^{m-1}\;A_\beta(z^2)\;\Phi_{2,2\gamma+1}(z) \;=\\ \\
=\;A_{0}(z^2)\;\Phi_{2,1}(z)\;+\;\sum_{\gamma=0}^{\frac{m}{2}} \;A_{\frac{m}{2}-\gamma}(z^2)\;\Phi_{2,2\gamma+1}(z)\;+\;\sum_{\gamma=1}^{m-1} \;A_{m-\gamma}(z^2)\;\Phi_{2,2\gamma+1}(z)
+\;\sum_{\gamma=\frac{m}{2}+1}^{m-1} \;A_{\frac{3m}{2}-\gamma}(z^2)\;\Phi_{2,2\gamma+1}(z) \\ \\
=\;\sum_{\gamma=0}^{m-1} \;\bigg(\;A_{\frac{m}{2}-\gamma}(z^2)\;+\;A_{m-\gamma}(z^2)\;\bigg)\;\Phi_{2,2\gamma+1}(z),
	\end{array}\]
	where the last equality holds since $A_{0}(z)\;=\;A_{m}(z)$ and $A_{\frac{3m}{2}-\gamma}(z)\;=\;A_{\frac{m}{2}-\gamma}(z)$. Equation \eqref{eq:dual_char_even} follows then from \eqref{eq:lemma}.
	
	Conversely, if \eqref{eq:dual_char_even} holds, we can subtract it from the necessary condition for refinability \eqref{eq:sym_char} with $\tau=\frac{1}{2}$ and $T=2$, getting again \eqref{eq:LHS_odd}, while, for the right-hand-side,
	\begin{equation} \label{eq:RHS_even}\begin{array}{l}
		m\;z^{-1}\;\sum_{\beta=0}^{m-1}\;\sum_{\substack{\gamma=0\\\gamma+2\beta \equiv_m 1}}^{2m-1}\;A_{\beta}(z^2)\;\Phi_{2,\gamma}(z)\;-\;m\;z^{-1}\;\sum_{\gamma=0}^{m-1} \;\bigg(\;A_{\frac{m}{2}-\gamma}(z^2)\;+\;A_{m-\gamma}(z^2)\;\bigg)\;\Phi_{2,2\gamma+1}(z)\;=\\ \\
		\qquad\qquad=\;m\;z^{-1}\; \Bigg( \;\sum_{\substack{\beta,\gamma=0\\2(\gamma+\beta) \equiv_m 1}}^{m-1}\;A_{\beta}(z^2)\;\Phi_{2,2\gamma}(z)\;+\;\sum_{\substack{\beta,\gamma=0\\2(\gamma+\beta) \equiv_m 0}}^{m-1}\;A_{\beta}(z^2)\;\Phi_{2,2\gamma+1}(z) \\ \\
		\hfill -\;\sum_{\gamma=0}^{m-1} \;\bigg(\;A_{\frac{m}{2}-\gamma}(z^2)\;+\;A_{m-\gamma}(z^2)\;\bigg)\;\Phi_{2,2\gamma+1}(z)\; \Bigg ) \\ \\
		\qquad\qquad=\;0.
	\end{array}\end{equation}
	Thus we obtain again \eqref{eq:step3} which leads to $\varphi(k)\;=\;\delta_{0,k}$, $k\in\mathbb{Z}$, so concluding the proof.
\end{proof}

\section{Constructive Examples} \label{sec:dual_interp}


In this section we show a linear algebra approach exploiting \eqref{eq:dual_char_odd} and \eqref{eq:dual_char_even} for the construction of dual interpolatory schemes  with arity $3$, $4$ and $5$, pointing out the pros and cons with respect to known primal interpolatory schemes. Exploiting the various degrees of freedom, this method can be used to search for dual interpolatory schemes with given wanted properties such as given arity, values of the basic limit function $\left \{\varphi \left ( k+ \frac{1}{2} \right ) \right \}_{k\in\mathbb{Z}}$, specific support length of the mask/basic limit function and/or degree of polynomial reproduction. The cases when the resulting linear system has no solution, are equivalent to the non-existence of a dual interpolatory scheme that satisfies the required properties.

In what follows matrices and vectors are considered to be bi-infinite where not specified and are labeled by bold uppercase and lowercase letters respectively. When useful, subsets of those matrices and vectors will be denoted using the convenient MatLab notation. Furthermore, we let the reader know that the given estimates of the regularity of the proposed schemes are obtained via joint spectral radius techniques. For the approximation of the joint spectral radius, the MatLab package \emph{t-toolboxes}, with parameter $\delta=1-1^{-12}$, has been used. This package implements the modified invariant polytope method by Guglielmi, Mejstrik and Protasov (see \cite{MR3886713,MR3009529,THOMAS2}). The H\"older regularities are given with a precision of $10^{-4}$.


\subsection{General Strategy}

The linear system to be solved arises from matching the coefficients of the Laurent polynomial in the left- and right-hand side of \eqref{eq:dual_char_odd} (when $m$ is odd) or \eqref{eq:dual_char_even} (when $m$ is even). The result is described in the following Proposition.

\begin{prop}
	Let $m\in\mathbb{N}$, $m>3$. Necessary condition for a bi-infinite vector $\mathbf{a}$ to be the mask of an $m$-ary dual interpolatory scheme with basic limit function $\varphi$ with prescribed values $\left \{ \varphi \left ( \frac{2k+1}{2} \right ) \right \}_{k\in\mathbb{Z}}$ is to be the solution of the linear system
	\begin{equation} \label{eq:lin_sys}
		\mathbf{M}\;\mathbf{a}\;=\;\mathbf{c},
	\end{equation}
	where
	\begin{equation} \label{eq:mat_M}
		\mathbf{M}(\alpha,\beta)\;=\;\varphi\left(\frac{m\alpha+1}{2}-\beta\right),\quad\alpha,\beta\in\mathbb{Z},
	\end{equation}
	and
	\begin{equation} \label{eq:vec_C}
		\mathbf{c}(\alpha)\;=\;\varphi\left(\frac{\alpha}{2}\right),\quad \alpha\in\mathbb{Z}.
	\end{equation}
\end{prop}

\begin{proof}
	We can rewrite \eqref{eq:dual_char_odd} and \eqref{eq:dual_char_even} in the following form
	\begin{equation} \label{eq:lin_sys_1}
		\mathbf{z}^T\;\widetilde{\mathbf{M}}\;\mathbf{a}\;=\;\mathbf{z}^T\;\widetilde{\mathbf{c}},
	\end{equation}
	where
	\[
		\mathbf{z}^T\;=\;\begin{bmatrix} &\dots,& z^{-1}, & 1, & z, & \dots& \end{bmatrix}.
	\]
	
	Since the left-hand sides of \eqref{eq:dual_char_odd} and \eqref{eq:dual_char_even} coincide, $\widetilde{\mathbf{c}}$ is the same in both cases. In particular,
	\[\begin{array}{rcl} \frac{1}{2}\;+\;\sum_{\gamma=0}^{m-1}\;\Phi_{2,2\gamma+1}(z^{m})&=&\frac{1}{2}\;+\;\frac{1}{2}\;\sum_{\gamma=0}^{m-1}\;\sum_{k\in\mathbb{Z}}\;\varphi\left(mk+\frac{2\gamma+1}{2}\right)\;z^{m(2mk+2\gamma+1)}\\ \\
		&=&\frac{1}{2}\;+\;\frac{1}{2}\;\sum_{\alpha\in m(2\mathbb{Z}+1)}\;\varphi\left(\frac{\alpha}{2m}\right)\;z^\alpha,
	\end{array}\]
	which leads to $\widetilde{\mathbf{c}}=[c_\alpha]_{\alpha\in\mathbb{Z}}$ with
	\begin{equation} \label{eq:constant_term}
		c_\alpha\;=\;\left\{\begin{array}{cl}
			\frac{1}{2},&\textrm{ if } \alpha=0,\\ \\
			\frac{1}{2}\varphi\left(\frac{\alpha}{2m}\right),&\textrm{ if } \alpha\in m(2\mathbb{Z}+1),\\ \\
			0,&\textrm{ otherwise.}
		\end{array}\right.
	\end{equation}
	
	For the matrix $\widetilde{\mathbf{M}}$ we have to consider two different cases: $\mathbf{M}_{odd}$ and $\mathbf{M}_{even}$,  respectively for \eqref{eq:dual_char_odd} and \eqref{eq:dual_char_even}. Starting from the case $m\in2\mathbb{N}+1$, the right-hand side of \eqref{eq:dual_char_odd} leads to
	\[\scalebox{0.9}{$\begin{array}{l}
		m\;z^{-1}\;\left(\;\frac{A_{\frac{m+1}{2}}(z^2)}{2}\;+\;\sum_{\gamma=0}^{m-1} \;A_{m-\gamma}(z^2)\;\Phi_{2,2\gamma+1}(z)\;\right)\;=\; \\ \\
		\qquad =\;mz^{-1}\;\left(\;\frac{1}{2m}\;\sum_{k\in\mathbb{Z}}\;a_{mk+\frac{m+1}{2}}\;z^{2mk+m+1}\;+\;\frac{1}{2m}\;\sum_{\gamma=0}^{m-1} \;\sum_{j,k\in\mathbb{Z}}\;a_{mk+m-\gamma}\;\varphi\left(mj+\frac{2\gamma+1}{2}\right)\;z^{2m(j+k+1)+1}\;\right) \\ \\
		\qquad=\;\frac{1}{2}\;\sum_{\alpha\in m(2\mathbb{Z}+1)}\;a_{\frac{\alpha+1}{2}}\;z^{\alpha}\;+\;\frac{1}{2}\;\sum_{\alpha\in2m\mathbb{Z}}\;\sum_{\beta\in\mathbb{Z}}\;a_{\beta}\;\varphi\left(\frac{\alpha+1}{2}-\beta\right)\;z^{\alpha},
	\end{array}$}\]
	which translates to
	\begin{equation} \label{eq:M_odd}
		\mathbf{M}_{odd}(\alpha,\beta)\;=\;\left\{\begin{array}{cl}
			\frac{1}{2},&\textrm{ if } \alpha\in m(2\mathbb{Z}+1) \textrm{ and } \beta=\frac{\alpha+1}{2},\\ \\
			\frac{1}{2}\;\varphi\left(\frac{\alpha+1}{2}-\beta\right),&\textrm{ if } \alpha\in 2m\mathbb{Z},\\ \\
			0,&\textrm{ otherwise,}
		\end{array}\right. \qquad \alpha,\beta\in\mathbb{Z}.
	\end{equation}
	
	Instead, for $m\in2\mathbb{N}$, from the right-hand side of \eqref{eq:dual_char_even}, we get
	\[\scalebox{0.85}{$\begin{array}{l}
		m\;z^{-1}\;\sum_{\gamma=0}^{m-1} \;\left( \;A_{\frac{m}{2}-\gamma}(z^2)\;+\;A_{m-\gamma}(z^2)\;\right) \;\Phi_{2,2\gamma+1}(z)\;= \\ \\
		\qquad =\;\frac{1}{2}\;\sum_{\gamma=0}^{m-1} \;\sum_{j,k\in\mathbb{Z}}\;a_{mk+\frac{m}{2}-\gamma}\;\varphi\left(mj+\frac{2\gamma+1}{2}\right)\;z^{m(2j+2k+1)}\;+\;\frac{1}{2}\;\sum_{\gamma=0}^{m-1} \;\sum_{j,k\in\mathbb{Z}}\;a_{mk+m-\gamma}\;\varphi\left(mj+\frac{2\gamma+1}{2}\right)\;z^{2m(j+k+1)}\\ \\
		\qquad=\;\frac{1}{2}\;\sum_{\alpha\in m(2\mathbb{Z}+1)}\;\sum_{\beta\in\mathbb{Z}}\;a_{\beta}\;\varphi\left(\frac{\alpha+1}{2}-\beta\right)\;z^{\alpha}\;+\;\frac{1}{2}\;\sum_{\alpha\in2m\mathbb{Z}}\;\sum_{\beta\in\mathbb{Z}}\;a_{\beta}\;\varphi\left(\frac{\alpha+1}{2}-\beta\right)\;z^{\alpha} \\ \\
		\qquad=\;\frac{1}{2}\;\sum_{\alpha\in m\mathbb{Z}}\;\sum_{\beta\in\mathbb{Z}}\;a_{\beta}\;\varphi\left(\frac{\alpha+1}{2}-\beta\right)\;z^{\alpha}.
	\end{array}$}\]
	Thus,
	\begin{equation} \label{eq:M_even}
		\mathbf{M}_{even}(\alpha,\beta)\;=\;\left\{\begin{array}{cl}
			\frac{1}{2}\;\varphi\left(\frac{\alpha+1}{2}-\beta\right),&\textrm{ if } \alpha\in m\mathbb{Z},\\ \\
			0,&\textrm{ otherwise,}
		\end{array}\right. \qquad \alpha,\beta\in\mathbb{Z}.
	\end{equation}
	
	Since \eqref{eq:lin_sys_1} must hold for every unitary $z$, the mask $\mathbf{a}$ must satisfy the linear system
	\begin{equation} \label{eq:lin_sys_2}
		\widetilde{\mathbf{M}}\;\mathbf{a}\;=\;\widetilde{\mathbf{c}}.
	\end{equation}
	Now we observe that all the rows in \eqref{eq:constant_term}, \eqref{eq:M_odd} and \eqref{eq:M_even} with index $\alpha\not\in m\mathbb{Z}$ are negligible since they contain only zero. Thus, recalling the interpolation property of $\varphi$, we can reduce \eqref{eq:lin_sys_2} to \eqref{eq:lin_sys}, where we eliminated the factor $\frac{1}{2}$ from both sides of \eqref{eq:lin_sys_2}.
\end{proof}

\begin{rem}
	The linear system \eqref{eq:lin_sys} coincides with the evaluations of the refinement equation \eqref{eq:ref_eq} at the points $\frac{\alpha}{2}$, $\alpha\in\mathbb{Z}$. Note that, since $\varphi$ is compactly supported, the matrix $\mathbf{M}$ is a band limited matrix.
\end{rem}

\begin{rem}
	If $m\in 2\mathbb{N}+1$, it follows from \eqref{eq:lin_sys_2}, \eqref{eq:constant_term} and the first line of \eqref{eq:M_odd} that the samples of $\varphi$ appear directly as elements of the mask $\mathbf{a}$.
\end{rem}

The linear system in \eqref{eq:lin_sys}, however, only encodes equations \eqref{eq:dual_char_odd} and \eqref{eq:dual_char_even}, without giving any clue about the eventual convergence of the scheme associated to the mask $\mathbf{a}$. The first thing we can do in this direction is to impose the necessary condition for convergence (see, e.g., \cite{MR1079033}), i.e.
\begin{equation} \label{eq:nec_cond}	
	A_\gamma(1)\;=\;\frac{1}{m},\quad \gamma=1,\dots,m
	\quad\Longleftrightarrow\quad
	\sum_{k\in\mathbb{Z}}\;a_{mk+\gamma}\;=\;1, \quad \gamma=1,\dots,m.
\end{equation}
We can then enlarge the system \eqref{eq:lin_sys} with $m$ rows to include \eqref{eq:nec_cond} in the following way:
\begin{equation} \label{eq:lin_sys_4}
	\begin{bmatrix}\mathbf{M}\\\mathbf{N}\end{bmatrix}\;\mathbf{a}\;=\;\begin{bmatrix}\mathbf{c}\\\mathbf{1}\end{bmatrix},
\end{equation}
where
\[
	\mathbf{N}(\alpha,\beta)\;=\;\left\{\begin{array}{cl}
		1,&\textrm{ if } \beta\equiv_m \alpha,\\ \\
		0,&\textrm{ otherwise,}
	\end{array}\right. \quad \alpha=1,\dots,m,\;\beta\in\mathbb{Z}.
\]

When talking about interpolatory schemes, recalling \eqref{eq:symbol}, it is well known that the condition
\begin{equation} \label{eq:A_factorization}
	A(z)\;=\;\left(\frac{1+\dots+z^{m-1}}{m}\right)^d\;B(z),
\end{equation}
for some $d\in\mathbb{N}$ and Laurent polynomial $B(z)=\sum_{k\in\mathbb{Z}}b_k z^k$, is, on one hand, necessary for $\varphi$ to belong to $\mathcal{C}^{d-1}(\mathbb{R})$ and, on the other hand, sufficient for the scheme to reproduce any arbitrary polynomial $\pi$ of degree $d-1$, i.e. to satisfy
\[
	\pi(x)\;=\;\sum_{k\in\mathbb{Z}}\;\pi(k)\;\varphi(x-k),\quad x\in\mathbb{R},\; \; \pi\in\mathbb{P}_{d-1}(\mathbb{R}),
\]
see, e.g., \cite{MR2775138,MR2843037,MR3071114}. These requirements are rather fundamental and can be easily integrated in the system \eqref{eq:lin_sys_4}. Given $d\in\mathbb{N}$, due to \eqref{eq:A_factorization}, we can shift our goal from computing $\mathbf{a}$ to computing $\mathbf{b}$ solving the following linear system
\begin{equation} \label{eq:lin_sys_5} \frac{1}{m^{d-1}}\;\begin{bmatrix}\mathbf{M}\\\mathbf{N}\end{bmatrix}\;\mathbf{O}^d\;\mathbf{b}\;=\;\begin{bmatrix}\mathbf{c}\\\mathbf{1}\end{bmatrix},
\end{equation}
where
\begin{equation} \label{eq:mat_O}
	\mathbf{O}(\alpha,\beta)\;=\;\left\{\begin{array}{cl}
		1,&\textrm{ if } 0\leq\alpha-\beta\leq m-1,\\ \\
		0,&\textrm{ otherwise,}
	\end{array}\right. \quad \alpha,\beta\in\mathbb{Z}.
\end{equation}

Now we are left with two last degrees of freedom:
\begin{itemize}
	\item[i)] the choice of $k^*\in\mathbb{N}$ such that $\supp(\mathbf{a})=\{1-k^*,\dots,k^*\}$, which we recall is linked to $supp(\varphi)$ in the following way:
	\begin{equation} \label{eq:supp_phi}
		\supp(\varphi)\;=\;\left[\;\frac{1-2k^*}{2(m-1)},\;\frac{2k^*-1}{2(m-1)}\;\right],
	\end{equation}
	see, e.g., \cite{MR2775138};
	\item[ii)] the choice of the values $\varphi \left( \frac{2k+1}{2} \right)$ for $k\in\mathbb{Z}$.
\end{itemize}

Choosing $k^*\in\mathbb{N}$ in i) allows us to cut the infinite linear system in \eqref{eq:lin_sys_5} into a finite one.

\begin{prop} \label{prop:size}
	Let $k^*\in\mathbb{N}$. The bi-infinite linear system in \eqref{eq:lin_sys_5} is actually of dimension
	\[
		\left( \;\left\lfloor\;\frac{2k^*-1}{m-1}\;\right\rfloor\;-\;\left\lceil\;\frac{1-2k^*}{m-1}\;\right\rceil\;+\;m\;+\;1\;\right) \;\times\; \Big ( \; 2k^*\;-\;d(m-1)\; \Big ).
	\]
	In particular, \eqref{eq:lin_sys_5} is equivalent to the cut given by
	\begin{equation} \label{eq:red_lin_sys}
		\frac{1}{m^{d-1}}\;\begin{bmatrix} (\mathbf{M}\mathbf{O}^d)(\;\alpha_\ell:\alpha_r\;,\;\beta_\ell:\beta_r\;)\\ (\mathbf{N}\mathbf{O}^d)(\;:\;,\;\beta_\ell:\beta_r\;)\end{bmatrix}\;\mathbf{b}(\;\beta_\ell:\beta_r\;)\;=\;\begin{bmatrix} \mathbf{c}(\;\alpha_\ell:\alpha_r\;) \\ \mathbf{1}\end{bmatrix},
	\end{equation}
	where
	\[
		\alpha_\ell\;:=\;\left\lceil\;\frac{1-2k^*}{m-1}\;\right\rceil,\quad \alpha_r\;:=\;\left\lfloor\;\frac{2k^*-1}{m-1}\;\right\rfloor,\quad \beta_\ell\;:=\;1-k^*,\quad \beta_r\;:=\;k^*-d(m-1).
	\]	
\end{prop}

\begin{proof}
	From \eqref{eq:A_factorization} it follows straightforwardly that, since the first and last non-zero elements of $\mathbf{a}$ are $a_{1-k*}$ and $a_{k^*}$ respectively, the first and last non-zero elements of $\mathbf{b}$ have to be $b_{1-k*}$ and $b_{k^*-d(m-1)}$ respectively. Thus we just need to consider the columns of the linear system with indices from $\beta_\ell=1-k^*$ to $\beta_r=k^*-d(m-1)$.
	
	Now, for the rows, we have to cut horizontally the system such that all the non-zero elements of the considered columns are preserved. In particular we need to know the support of the columns of $\mathbf{MO}^d$ and the support of $\mathbf{c}$. Starting from $\mathbf{c}$, which is straightforward, due to \eqref{eq:vec_C} and \eqref{eq:supp_phi} we have that
	\[
		\supp(\mathbf{c})\;\subseteq\;\left\{\; \left\lceil\frac{1-2k^*}{m-1}\right\rceil ,\;\dots,\; \left\lfloor\frac{2k^*-1}{m-1}\right\rfloor \;\right\}\;=\;\{\;\alpha_\ell,\;\dots,\;\alpha_r\;\}.
	\]
	On the other hand, from \eqref{eq:mat_M} and \eqref{eq:mat_O} we have that, for $\beta\in\mathbb{Z}$,
	\[
		\supp(\;\mathbf{M}(\;:\;,\;\beta\;)\;)\;\subseteq\;\left\{\;\left\lceil\frac{1-2k^*}{m(m-1)}+\frac{2\beta-1}{m}\right\rceil ,\;\dots,\; \left\lfloor\frac{2k^*-1}{m(m-1)}+\frac{2\beta-1}{m}\right\rfloor \;\right\},
	\]
	\[
		\supp(\;\mathbf{O}^d(\;:\;,\;\beta\;)\;)\;=\;\left\{\;\beta\;,\;\beta+d(m-1)\;\right\},
	\]
	and thus
	\[
		\supp(\;\mathbf{MO}^d(\;:\;,\;\beta\;)\;)\;\subseteq\;\left\{\;\left\lceil\frac{1-2k^*}{m(m-1)}+\frac{2\beta-1}{m}\right\rceil ,\;\dots,\; \left\lfloor\frac{2k^*-1}{m(m-1)}+\frac{2\beta+2d(m-1)-1}{m}\right\rfloor \;\right\}.
	\]
	Therefore,
	\[\begin{array}{rcl}
		\supp(\;\mathbf{MO}^d(\;:\;,\;\beta_\ell\;:\;\beta_r\;)\;)&\subseteq&\left\{\;\left\lceil\frac{1-2k^*}{m(m-1)}+\frac{2\beta_\ell-1}{m}\right\rceil ,\;\dots,\; \left\lfloor\frac{2k^*-1}{m(m-1)}+\frac{2\beta_r+2d(m-1)-1}{m}\right\rfloor \;\right\}\\ \\
		&=&\left\{\; \left\lceil\frac{1-2k^*}{m-1}\right\rceil ,\;\dots,\; \left\lfloor\frac{2k^*-1}{m-1}\right\rfloor \;\right\}\;=\;\{\;\alpha_\ell,\;\dots,\;\alpha_r\;\},
	\end{array}\]
	and this concludes the proof.
\end{proof}

As a consequence of Proposition \ref{prop:size}, the linear system \eqref{eq:lin_sys_5} is, in general, rectangular and can present all possible scenarios: no solutions, unique solution or a family of solutions described by one or more parameters. However, both the number of equations of the system and the number of unknowns can be halved by requiring the mask $\mathbf{a}$ (and thus $\mathbf{b}$) to be symmetric.

Last but not least is the choice ii). This is influenced by $d$. Indeed, since the polynomial reproduction must hold even if we sample $\pi$ at $\mathbb{Z}/2$ rather than at $\mathbb{Z}$ (the set of polynomials of degree up to $d-1$ is closed under dilations), we can obtain a subdivision scheme with the desired reproduction property only if we choose as $\varphi \left( \frac{2k+1}{2} \right)$, $k\in\mathbb{Z}$, the same values of the basic limit function of a binary primal interpolatory scheme with equal or higher degree of polynomial reproduction. In what follows, we use the values given by the Dubuc-Deslauriers $2n$-point schemes, since they reach polynomial reproduction of degree $2n-1$ within the shortest possible support (see, e.g., \cite{MR982724,MR1184153}). Since Dubuc-Deslauriers schemes are symmetric and symmetry is a property usually required in many applications, we will always impose symmetry in all the following examples.

\subsection{A $3$-ary Dual Interpolatory Subdivision Scheme}

We start choosing $m=3$. As a good benchmark, the first construction is aimed at obtaining a subdivision scheme with basic limit function $\varphi\in\mathcal{C}^2(\mathbb{R})$ that reproduces cubic polynomials and has small support. To achieve that we choose the values of $\varphi$ at the half-integers to be the same as those of the Dubuc-Deslauriers $4$-point scheme, i.e.,
\begin{equation} \label{eq:DD4_samp}
	\left \{ \varphi \left( \frac{k}{2} \right) \right\}_{k=-3}^{3}\;=\;\frac{1}{16}\;\{-1, \, 0, \, 9, \, 16, \, 9, \, 0, \, -1\}
\end{equation}
and $0$ over the other half-integers. This implies that
\[
	\frac{2k^*-1}{m-1}\;=\;|\supp(\varphi)|\;\geq\;3\quad\Longrightarrow\quad k^*\;\geq\;\frac{7}{2}.
\]
Moreover, to get reproduction of cubic polynomials we need $d=4$ and thus, by \eqref{eq:A_factorization},
\[
	1-k^*\;<\;k^*-d(m-1) \quad\Longrightarrow\quad k^*\;>\;\frac{d(m-1)+1}{2}\;=\;\frac{9}{2}.
\]
It turns out that for $k^*\in\{5,6\}$ the associated linear system has no symmetric solutions. On the other hand, for $k^*=7$, i.e. $|\supp(\varphi)|=6.5$, the corresponding linear system
\[
	\begin{bmatrix}
		-\frac{1}{432},&   \frac{5}{432},&  \frac{35}{432} \smallskip \\
		\frac{1}{27},&    \frac{4}{27},&    \frac{10}{27} \smallskip\\
		\frac{1}{3},&     \frac{1}{2},&     \frac{2}{3} \smallskip\\
		\frac{26}{27},&   \frac{23}{27},&   \frac{17}{27} \smallskip\\
		\frac{289}{216},& \frac{211}{216},& \frac{109}{216} \smallskip\\
		2, & 2, & 2
	\end{bmatrix}\;\begin{bmatrix}
		b_1\\
		b_2\\
		b_3
	\end{bmatrix}\;=\;\begin{bmatrix}
		0 \smallskip\\
		-\frac{1}{16} \smallskip\\
		0 \smallskip\\
		\frac{9}{16} \smallskip\\
		1 \smallskip\\
		1
	\end{bmatrix},
\]
where
\[
	A(z)\;=\;\frac{1}{3}\;\sum_{k=-6}^7\;a_k\;z^k\;=\;z^{-6}\;\left(\frac{1+z+z^2}{3}\right)^4\;\left(b_3 + b_2 z + b_1 z^2 + b_1 z^3 + b_2 z^4 + b_3 z^5\right),
\]
has a unique symmetric solution which leads to the subdivision mask
\begin{equation} \label{eq:3ary_mask}
	 \scalebox{0.95}{$
	 \{a_k\}_{k=-6}^{7} \;=\; \scalebox{0.8}{$
		\left \{
			\frac{13}{1296}, \, -\frac{11}{648}, \, -\frac{1}{16}, \, -\frac{107}{1296}, \, \frac{179}{1296}, \, \frac{9}{16}, \, \frac{137}{144}, \, \frac{137}{144}, \, \frac{9}{16}, \, \frac{179}{1296}, \, -\frac{107}{1296}, \, -\frac{1}{16}, \, -\frac{11}{648}, \, \frac{13}{1296}
		\right \}.
	$}
	$}
\end{equation}

It can be shown that the corresponding refinable function $\varphi$, depicted in Figure \ref{fig:3ary}, belongs to $\mathcal{C}^{2.2760}(\mathbb{R})$.
With respect to the binary Dubuc-Deslauriers $4$-point scheme, this loses the step-wise interpolation and has a basic limit function with a slightly wider support, but it gains in regularity and in a faster rate of convergence due to the higher arity. Even if there are three submasks, two of them are actually the same, one being the flipped version of the other, and they are rather small, with the biggest one having support $5$.
This subdivision scheme has also higher performance with respect to the primal Dubuc-Deslauriers $4$-point ternary scheme, which reproduces cubic polynomials but is only $\mathcal{C}^{1.8173}(\mathbb{R})$ \cite{MR3702925}, and with respect to the family of $4$-point schemes proposed in \cite{MR1879678}, for which the most regular scheme belongs to $\mathcal{C}^{2.1816}(\mathbb{R})$ but reproduces only quadratic polynomials. Note that the superior properties of the dual scheme are achieved by enlarging the support of the primal ternary subdivision mask with the inclusion of three additional elements only (cf. \eqref{eq:3ary_mask}).

\begin{figure}[h!]
	\centering
	\includegraphics[scale=0.5]{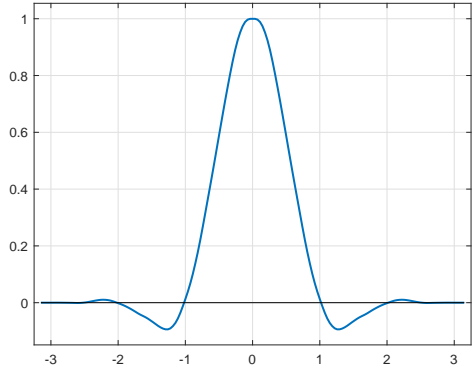}
	\caption{The basic limit function of the ternary dual interpolatory scheme corresponding to the mask in \eqref{eq:3ary_mask}, which is supported in $[-3.25,3.25]$, reproduces cubic polynomials and belongs to $\mathcal{C}^{2.2760}(\mathbb{R})$.}
\label{fig:3ary}
\end{figure}


\subsection{A One-Parameter Family of $5$-ary Dual Interpolatory Subdivision Schemes}

We now consider $m=5$. Sticking to the samples of the Dubuc-Deslauriers $4$-point scheme \eqref{eq:DD4_samp}, we now try to obtain a symmetric subdivision scheme with smaller support, e.g. $|\supp(\varphi)|\leq \frac{19}{4} =4.75$ which means a mask with at most $20$ entries, sacrificing a degree of polynomial reproduction. The corresponding linear system then happens to be rank-deficient and solving it leads to a family of masks depending on the parameter $w\in\mathbb{R}$:
\begin{equation} \label{eq:5ary_mask}
	\scalebox{0.83}{$\begin{array}{rcl}
		\{a_k\}_{k=-9}^{10}&=& \left\{\;\frac{w}{400},\; \frac{9\,w}{400}, \; -\frac{1}{16}, \; -\frac{9\,w}{400}-\frac{21}{200}, \; -\frac{w}{400}-\frac{9}{200}, \; \frac{11}{200}-\frac{3\,w}{400}, \; \frac{39}{200}-\frac{27\,w}{400}, \; \frac{9}{16}, \; \frac{27\,w}{400}+\frac{91}{100}, \; \frac{3\,w}{400}+\frac{99}{100},\right. \\ \\
		& & \left. \frac{3\,w}{400}+\frac{99}{100}, \; \frac{27\,w}{400}+\frac{91}{100}, \; \frac{9}{16}, \; \frac{39}{200}-\frac{27\,w}{400}, \; \frac{11}{200}-\frac{3\,w}{400}, \; -\frac{w}{400}-\frac{9}{200}, \; -\frac{9\,w}{400}-\frac{21}{200}, \; -\frac{1}{16}, \; \frac{9\,w}{400}, \; \frac{w}{400}\; \right\}.
		
	\end{array}$}
\end{equation}
Analysing the infinity norm of the difference scheme for three levels (cf. Figure \ref{fig:5ary_reg} and see, e.g., \cite{CHOI2006351,MR1172120,MR2008967} for details) gives the  sufficient conditions in Table \ref{tab_rangew} for $\varphi$ belonging to $\mathcal{C}^k(\mathbb{R})$.
\begin{table}[h!]
	\centering
	\begin{tabular}{|c|c|}
		\hline
		$k$ & range \\
		\hline
		0 & $-14.4545\;\leq\;w\;\leq\;   11.7273$ \\
		1 & $-4.1983\;\leq\;w\;\leq\;    1.4711$ \\
		2 & $-1.5832 \;\leq\;w\;\leq\;  -1.0187$ \\
		\hline
	\end{tabular}
	\caption{Sufficient conditions on the free parameter $w$ for $\varphi$ belonging to $\mathcal{C}^k(\mathbb{R})$.}
\label{tab_rangew}
\end{table}

\smallskip
Figure \ref{fig:5ary_reg_jsr} shows the function $\varphi$ obtained with $w=-1.4$, whose regularity is, as expected, greater than $2$ and approximately $2.1568$. \\ 
Summarizing, all the schemes described by the masks in \eqref{eq:5ary_mask}, with $-14.4545\leq w\leq11.7273$, are convergent, reproduce quadratic polynomials, but never cubics, and their basic limit functions attain the same values over $\mathbb{Z}/2$.
While not being top notch for what concerns polynomial reproduction, all the $\mathcal{C}^2(\mathbb{R})$ schemes of this family achieve a good regularity, together with the interpolatory property in a very short support, and also a really good looking shape, see e.g. Figure \ref{fig:5ary}. Moreover, the arity $5$ guarantees a very fast convergence.

\begin{figure}[h!]
	\centering
	\includegraphics[width=0.65\textwidth]{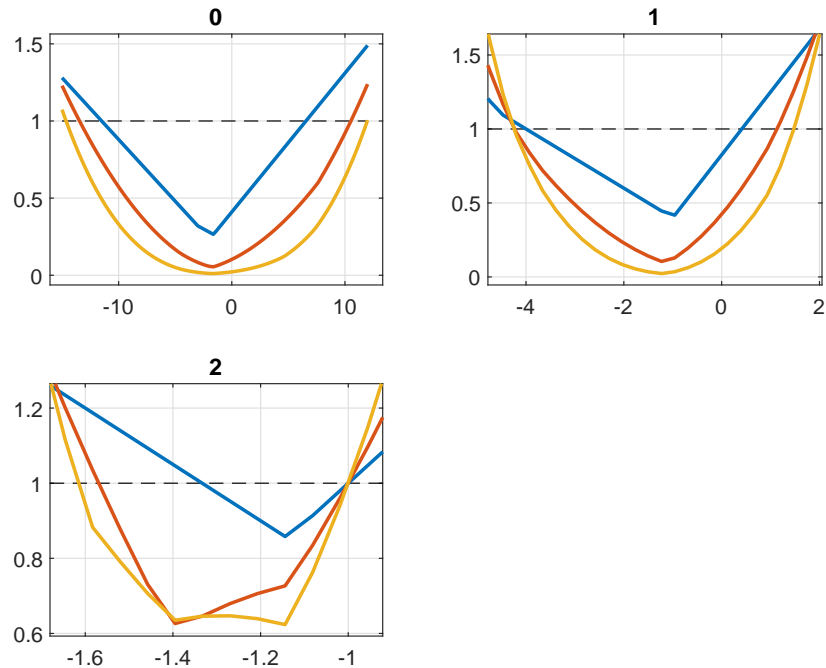}
	\caption{Regularity analysis of the family of $5$-ary masks in \eqref{eq:5ary_mask}. The three figures depict the analysis for $\mathcal{C}^0(\mathbb{R})$, $\mathcal{C}^1(\mathbb{R})$ and $\mathcal{C}^2(\mathbb{R})$ regularity, via estimates of the infinity norm of the first, second and third difference scheme respectively in blue, red and yellow. For each figure, on the $x$-axis we have the parameter $w$, while on the $y$-axis we have upper bounds for the infinity norm of the corresponding difference scheme given by one level, two level or three level of the difference matrix in blue, red and yellow respectively.}
	\label{fig:5ary_reg}
\end{figure}

\medskip
\begin{figure}[h!]
	\centering
	\includegraphics[width=0.38\textwidth]{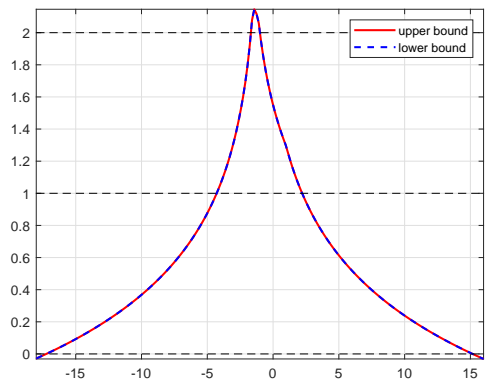}
	\caption{Regularity analysis via JSR of the family of $5$-ary masks in \eqref{eq:5ary_mask} for $w \in [-18, 16]$. The maximum regularity is approximately
	$2.1457$ achieved around $w=-1.4271$.}
	\label{fig:5ary_reg_jsr}
\end{figure}

\begin{figure}[h!]
	\centering
	\includegraphics[width=0.39\textwidth]{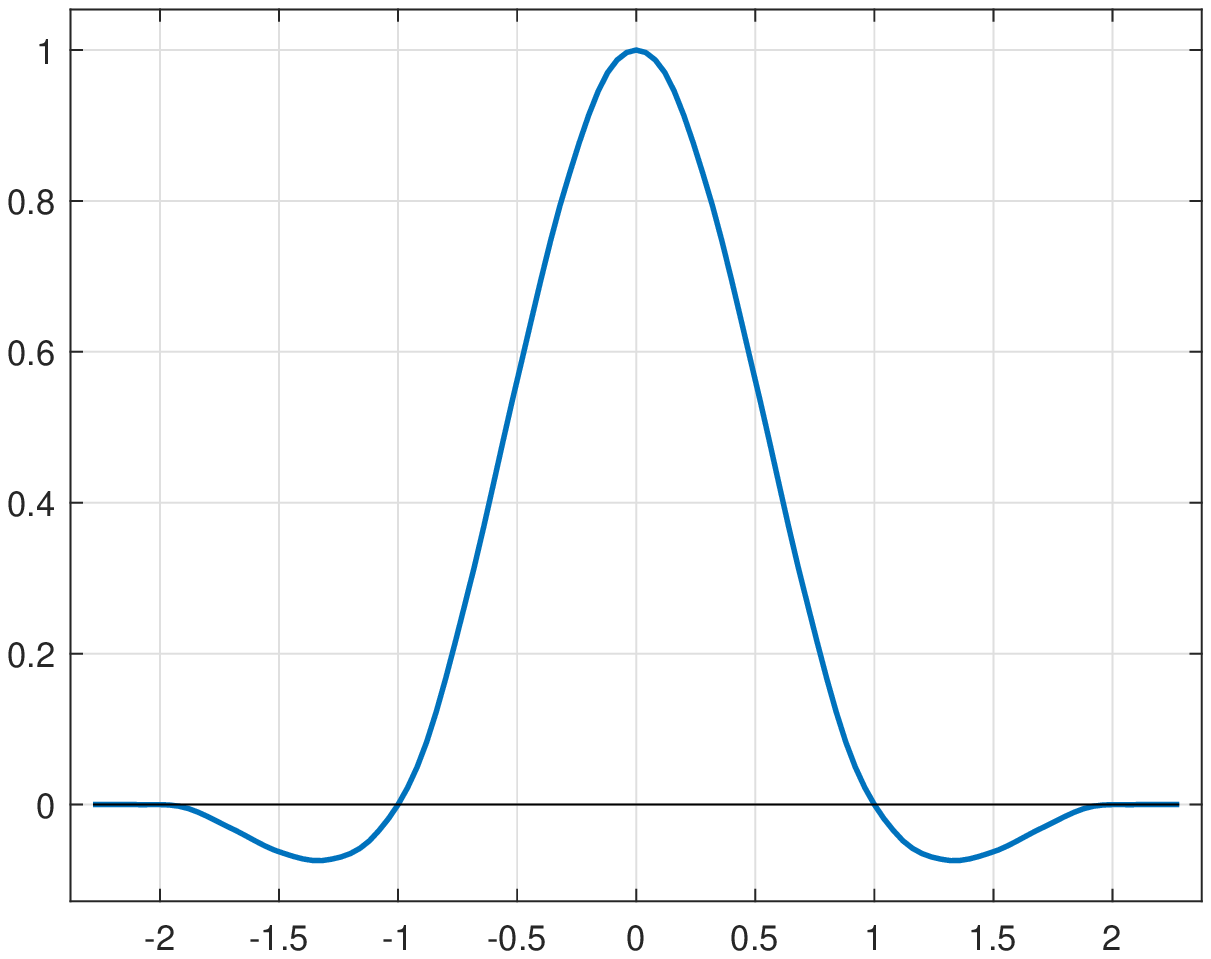}
	\caption{The basic limit function of the $5$-ary scheme corresponding to the mask in \eqref{eq:5ary_mask} with $w=-1.4$, which is supported in $[-2.375,2.375]$, belongs to $\mathcal{C}^{2.2044}(\mathbb{R})$ and reproduces polynomials up to degree $2$.}
	\label{fig:5ary}
\end{figure}

\subsection{A Three-Parameter Family of $4$-ary Dual Interpolatory Subdivision Schemes}

For $m=4$ we present a three-parameter family of masks. As for the support and the polynomial reproduction we ask for very mild requirements. In particular, we search for schemes that reproduce quadratic polynomials and that have $|\supp(\varphi)|\leq7$, i.e. $k^*\leq11$. However we introduce a parameter from the beginning, taking as samples of $\varphi$ over $\mathbb{Z}/2$, a convex combination between the values of the Dubuc-Deslauriers $4$-point scheme and the Dubuc-Deslauriers $6$-point scheme, i.e.
\begin{equation} \label{eq:DD4_DD6_samp}
	\scalebox{0.85}{$ \left\{ \varphi \left( \frac{k}{2} \right) \right\}_{k=-5}^{5}\;=\;
\left \{
		\frac{3w}{256}, \, 0, \, -\frac{9w}{256}-\frac{1}{16}, \, 0, \, \frac{3w}{128}+\frac{9}{16}, \, 1, \, \frac{3w}{128}+\frac{9}{16}, \, 0, \, -\frac{9w}{256}-\frac{1}{16}, \, 0, \, \frac{3w}{256}
	\right \},\quad w\in[0,1].$}
\end{equation}
\newpage
The solutions of the resulting linear system depend on two more parameters $v$ and $u$, giving rise to a family of masks $\{a_k\}_{k=-10}^{11}$ where
\begin{equation} \label{eq:4ary_mask}
\scalebox{0.88}{$\begin{array}{rcl}
	a_{-10}\;=\;a_{11}&=&-\frac{w\big[12(5w+8)v+4(9w + 16)u-3(155w + 48)\big]}{1024(w+24)(3w+4)},\\ \\
	a_{-9}\;=\;a_{10}&=&-\frac{(9w+16)\left[12(5w+8)v+4(9w + 16)u-3(155w + 48)\right]}{3072(w+24)(3w+4)},\\ \\
	a_{-8}\;=\;a_{9}&=&-\frac{6v+4u-9}{128},\\ \\
	a_{-7}\;=\;a_{8}&=&\frac{v}{64},\\ \\
	a_{-6}\;=\;a_{7}&=&\frac{12(63w^2 - 376w - 6784)v+4(99w^2 - 1344w - 20480)u-3(2307w^2 + 176w - 61440)}{3072(w+24)(3w+4)},\\ \\
	a_{-5}\;=\;a_{6}&=&\frac{12(117w^2 + 904w + 7168)v+4(225w^2 + 2352w + 21248)u-3(3633w^2 + 18112w + 75264)}{3072(w+24)(3w+4)},\\ \\
	a_{-4}\;=\;a_{5}&=&\frac{8v+6u-17}{64},\\ \\
	a_{-3}\;=\;a_{4}&=&\frac{u}{32},\\ \\
	a_{-2}\;=\;a_{3}&=&-\frac{3(27w^2 - 712w - 10240)v+(27w^2 - 2304w - 30848)u-3(441w^2 + 2102w - 19968)}{384(w+24)(3w+4)},\\ \\
	a_{-1}\;=\;a_{2}&=&-\frac{3(11w^2 + 296w + 3456)v+(27w^2 + 880w + 10368)u-(453w^2 + 9530w + 33888)}{128(w+24)(3w+4)},\\ \\
	a_{0}\;=\;a_{1}&=&-\frac{3(4v+4u-51)}{128}.
	\end{array}$}
\end{equation}

This family includes the two subdivision schemes presented in \cite{LUCIA}, which are obtained by selecting
\[
	w\;=\;0,\quad v\;=\;4\theta, \quad u\;=\;\frac{9}{4}\;-\;6\theta,
\]
and
$$
	w\;=\;\frac{256}{3} \theta, \quad
	v\;=\;\frac{520192 \theta^2 + 3936 \theta - 1077}{256(32 \theta + 5)}, \quad
	u\;=\;\frac{-(520192 \theta^2 - 2592 \theta - 2097)}{128(32 \theta + 5)},
$$
respectively. Thus they both turn out to be written in terms of a single free parameter $\theta$.
The first reproduces quadratic polynomials whereas the second reproduces cubic polynomials.

Note that the choice
\begin{equation} \label{eq:4ary_3rep}
	v\;=\;\frac{3(381w^2+246w-5744)}{512(3w+40)}
	\quad\textrm{ and }\quad
	u\;=\;\frac{9(-127w^2+54w+3728)}{256(3w+40)},
\end{equation}
guarantees cubic polynomial reproduction for any value of $w$. Instead, the reproduction of degree-$4$ polynomials is achieved only by one element of this family, the one having
\begin{equation} \label{eq:4ary_4rep}
	w\;=\;1,\quad v\;=\;-\frac{357}{512}\quad\textrm{ and }\quad u\;=\;\frac{765}{256},
\end{equation}
i.e. with the first half of the mask equal to
\begin{equation} \label{eq:5ary_mask_reg}
	\scalebox{0.88}{$\frac{1}{917504}\;\left \{
		2145, \, 17875, \, 8820, \, -9996, \, -39985, \, -127595, \, -66640, \, 85680, \, 325754, \, 739310, \, 899640 
	\right \}.$}	
\end{equation}
As expected, the maximal polynomial reproduction is achieved when $w=1$, i.e. when the samples at $\mathbb{Z}/2$ are taken from the Dubuc-Deslauriers $6$-point scheme, which reproduces also polynomials of degree $5$. However this is not achieved by the scheme with mask in \eqref{eq:5ary_mask_reg}. To get the reproduction of degree $5$ polynomials one has to allow a wider support for $\varphi$, namely $|\supp(\varphi)|\geq \frac{29}{3} =9.\overline{6}$.

Within this family of schemes, we observe a good trade off between regularity and polynomial reproduction, similarly to what shown,  e.g., in \cite{MR2474706,MR3474348}. Indeed, if we compare the scheme having mask satisfying \eqref{eq:4ary_3rep} with $w=0$ (Figure \ref{fig:4ary} $(a)$) and the one having mask satisfying \eqref{eq:4ary_4rep} (Figure \ref{fig:4ary} $(b)$), which are at the boundary with respect to the schemes with mask \eqref{eq:4ary_mask} that reproduce cubic polynomials, we have that the first is $\mathcal{C}^{2.3043}(\mathbb{R})$ but does not reproduce polynomials of degree $4$, while the second reproduces polynomial of degree $4$ but it is only $\mathcal{C}^{1.5761}(\mathbb{R})$. From the point of view of the shape, both limit functions have some undesirable oscillations which are much less evident in the scheme \eqref{eq:4ary_3rep} with $w=\frac{1}{2}$ (Figure \ref{fig:4ary} $(c)$) which is $\mathcal{C}^{2.2247}(\mathbb{R})$ and reproduces cubic polynomials. Thus the latter could be a good candidate for design applications.

\begin{figure}[h!]
	\begin{minipage}{0.5\textwidth}
		\includegraphics[width=0.9\textwidth]{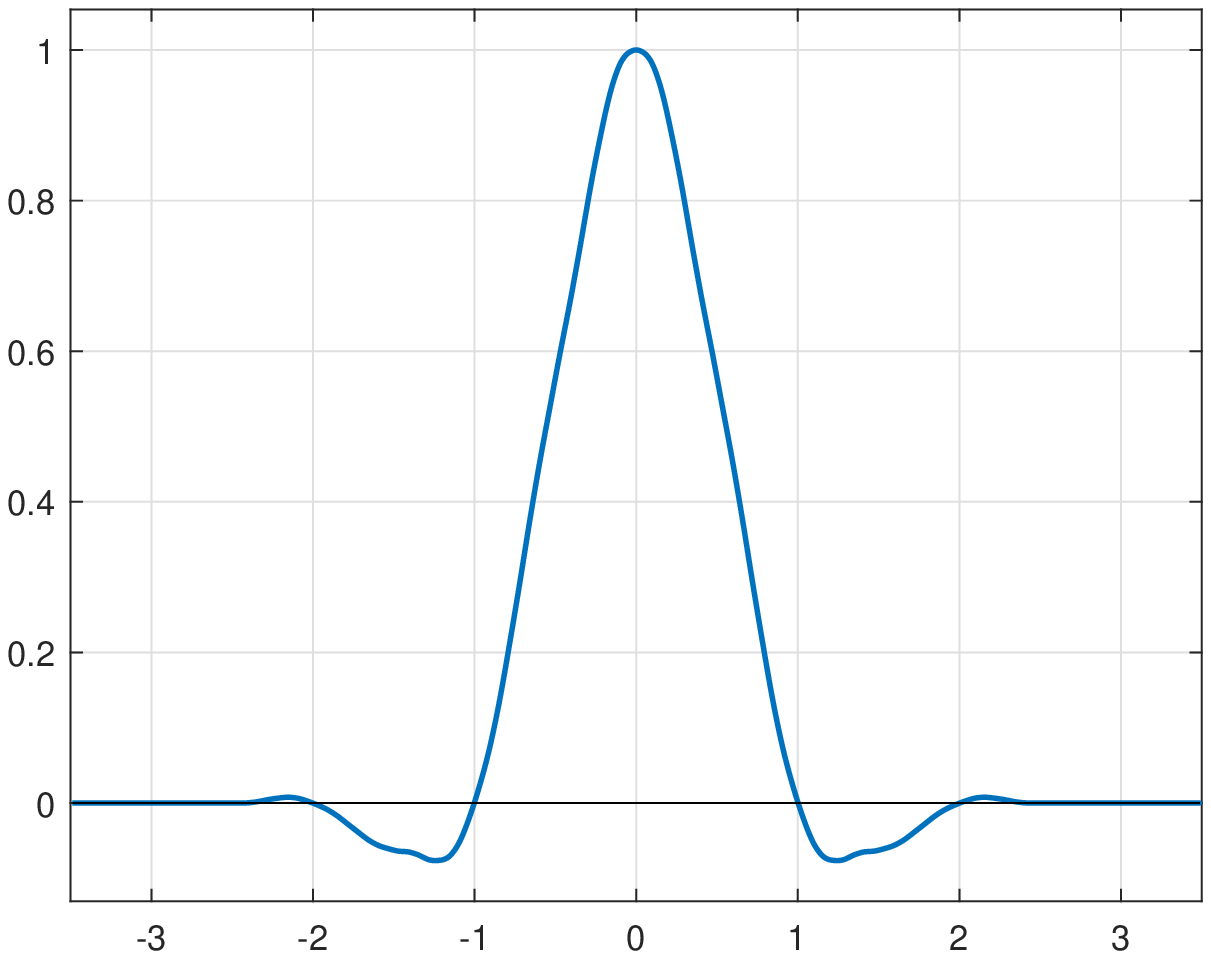}
		\begin{center}
			$(a)$
		\end{center}
	\end{minipage}
	\begin{minipage}{0.5\textwidth}
		\includegraphics[width=0.9\textwidth]{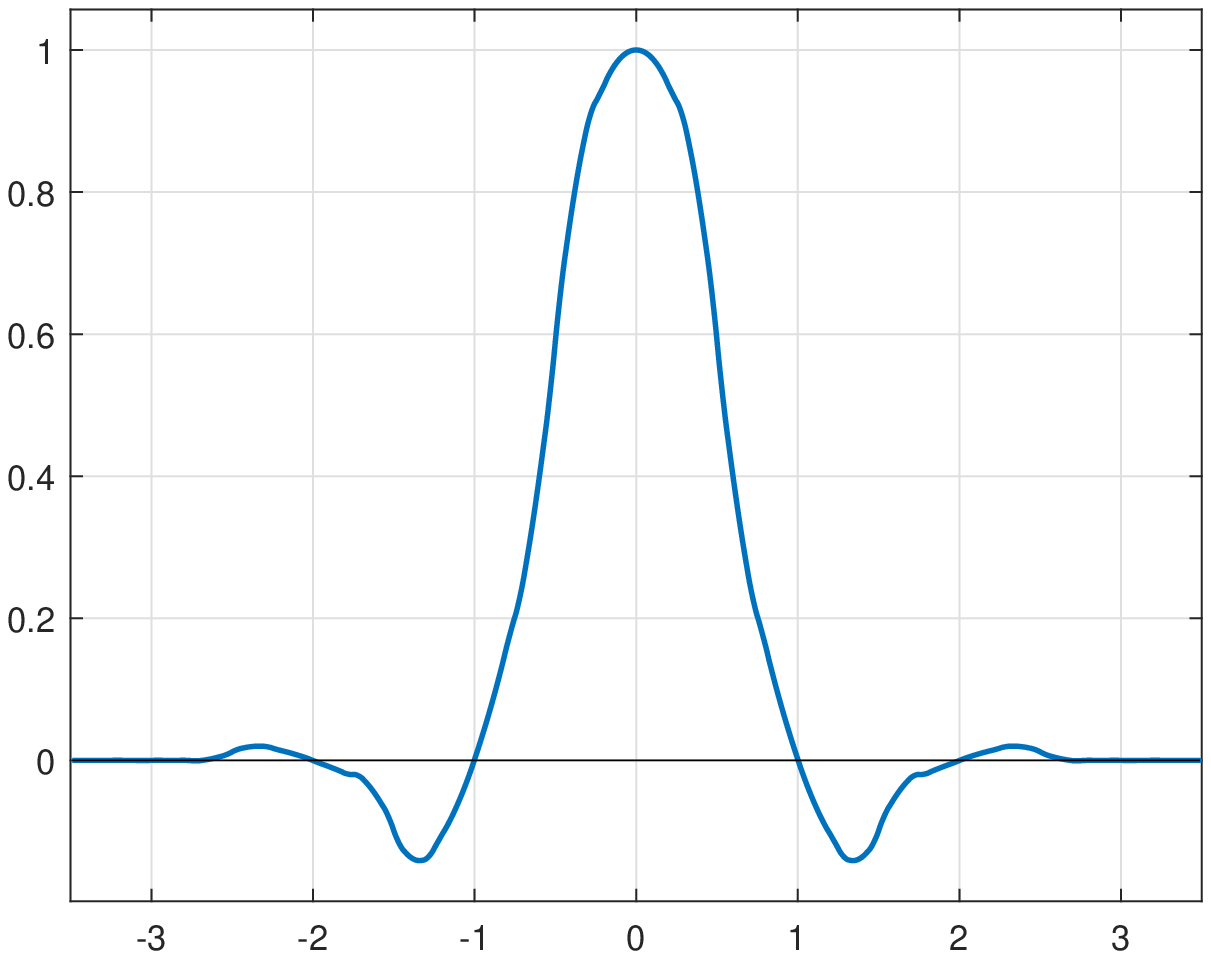}
		\begin{center}
			$(b)$
		\end{center}
	\end{minipage}
	$ $\\
	\begin{center}
		\includegraphics[width=0.45\textwidth]{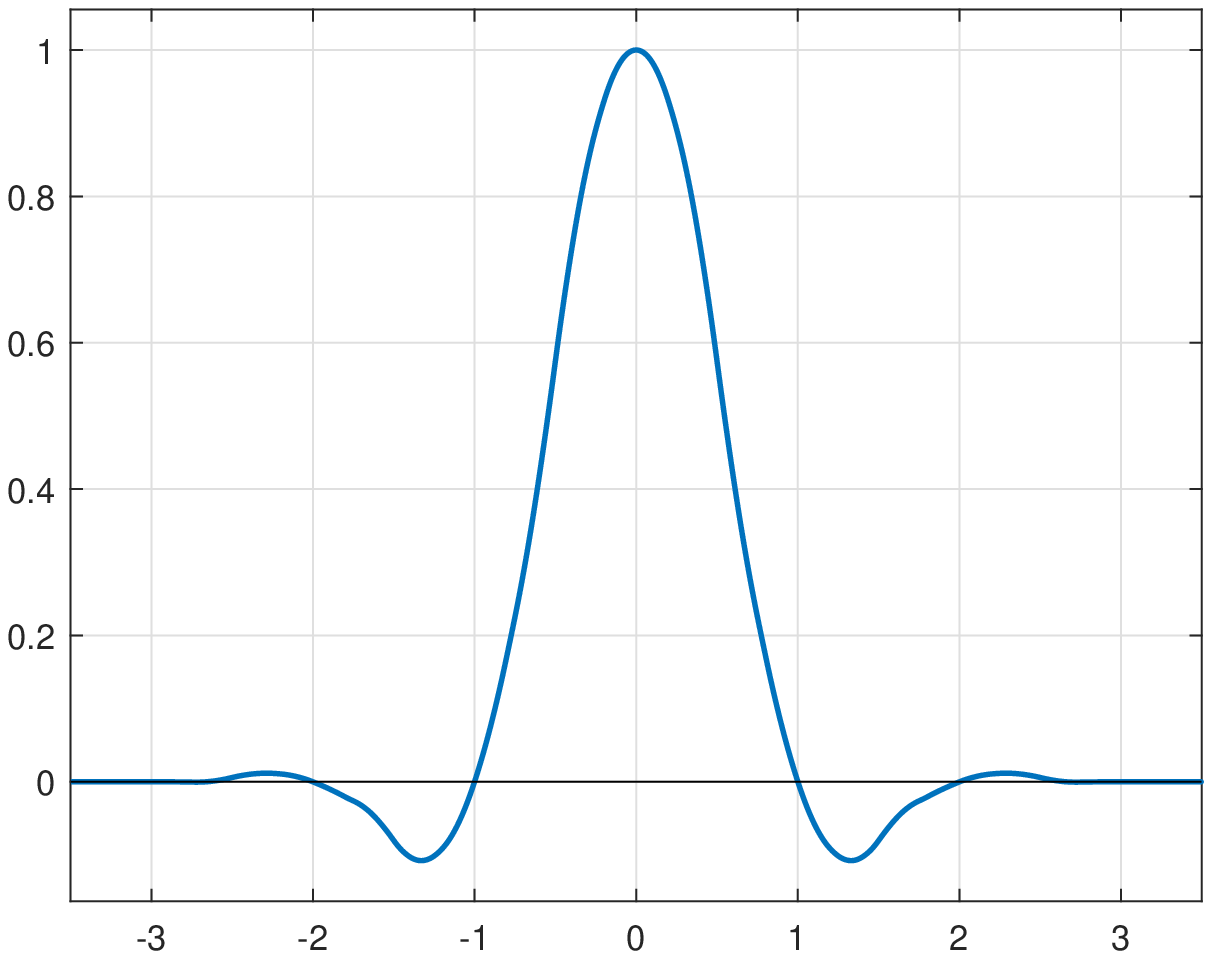}\\ $ $\\
		$\quad(c)$
	\end{center}
	\caption{$(a)$ The basic limit function of the $4$-ary dual interpolatory scheme corresponding to the mask in \eqref{eq:4ary_mask} satisfying \eqref{eq:4ary_3rep} with $w=0$, which is supported in $\left[- \frac{19}{6}, \frac{19}{6} \right]$, belongs to $\mathcal{C}^{2.3043}(\mathbb{R})$
	and reproduces polynomials up to degree $3$. $(b)$ The basic limit function of the $4$-ary dual interpolatory scheme corresponding to the mask in \eqref{eq:4ary_mask} satisfying \eqref{eq:4ary_4rep}, which is supported in $[-3.5,3.5]$, belongs to $\mathcal{C}^{1.5761}(\mathbb{R})$ 
	and reproduces polynomials up to degree $4$. $(c)$ The basic limit function of the $4$-ary dual interpolatory scheme corresponding to the mask in \eqref{eq:4ary_mask} satisfying \eqref{eq:4ary_3rep} with $w=\frac{1}{2}$, which is supported in $[-3.5,3.5]$, belongs to $\mathcal{C}^{2.2299}(\mathbb{R})$ 
	and reproduces polynomials up to degree $3$.}
	\label{fig:4ary}
\end{figure}

\begin{figure}[h!]
	\centering
	\includegraphics[width=0.4\textwidth]{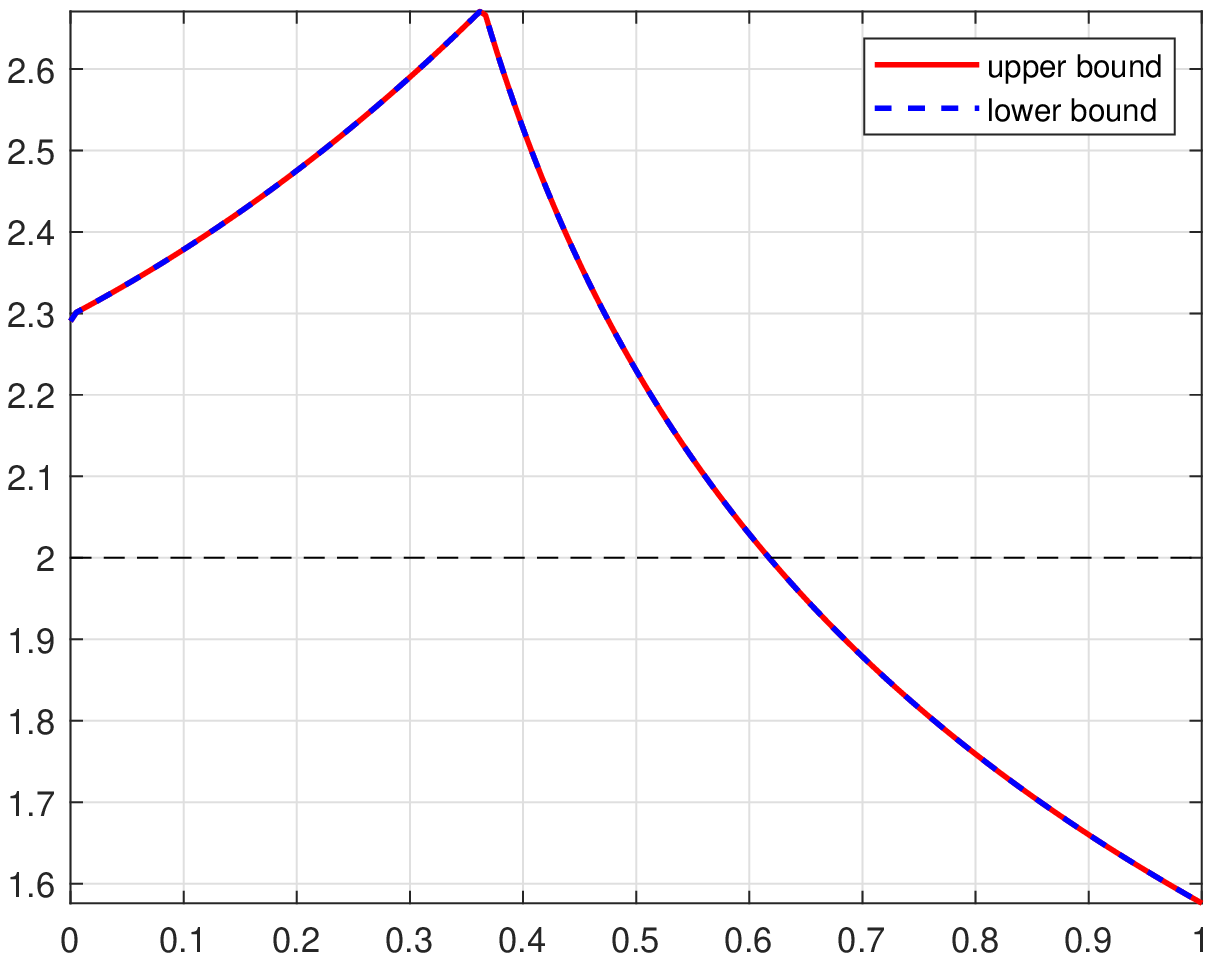}
	\caption{Regularity analysis via joint spectral radius of the family of $4$-ary masks in \eqref{eq:4ary_mask} with \eqref{eq:4ary_3rep} and $w \in [0, 1]$. The maximum regularity is approximately $2.6667$ achieved around $w=\frac{11}{30}$.}
	\label{fig:4ary_reg_jsr}
\end{figure}

\subsection{Qualitative Comparison}

We end with a direct comparison of the three best looking schemes among the ones introduced before, i.e., the ones whose basic limit functions are depicted in Figure \ref{fig:3ary}, \ref{fig:5ary} and \ref{fig:4ary} $(c)$ respectively. The test here is the interpolation of the vertices of a square and of a control polygon which presents different kind of angles (Figure \ref{fig:confront}). 
All these schemes are $\mathcal{C}^2(\mathbb{R})$. The $5$-ary scheme is the only one that does not reproduce cubic polynomials; however is the one that in both examples fits the control points with smaller oscillations.

\begin{figure}[h!]
	\centering
	\includegraphics[scale=0.69]{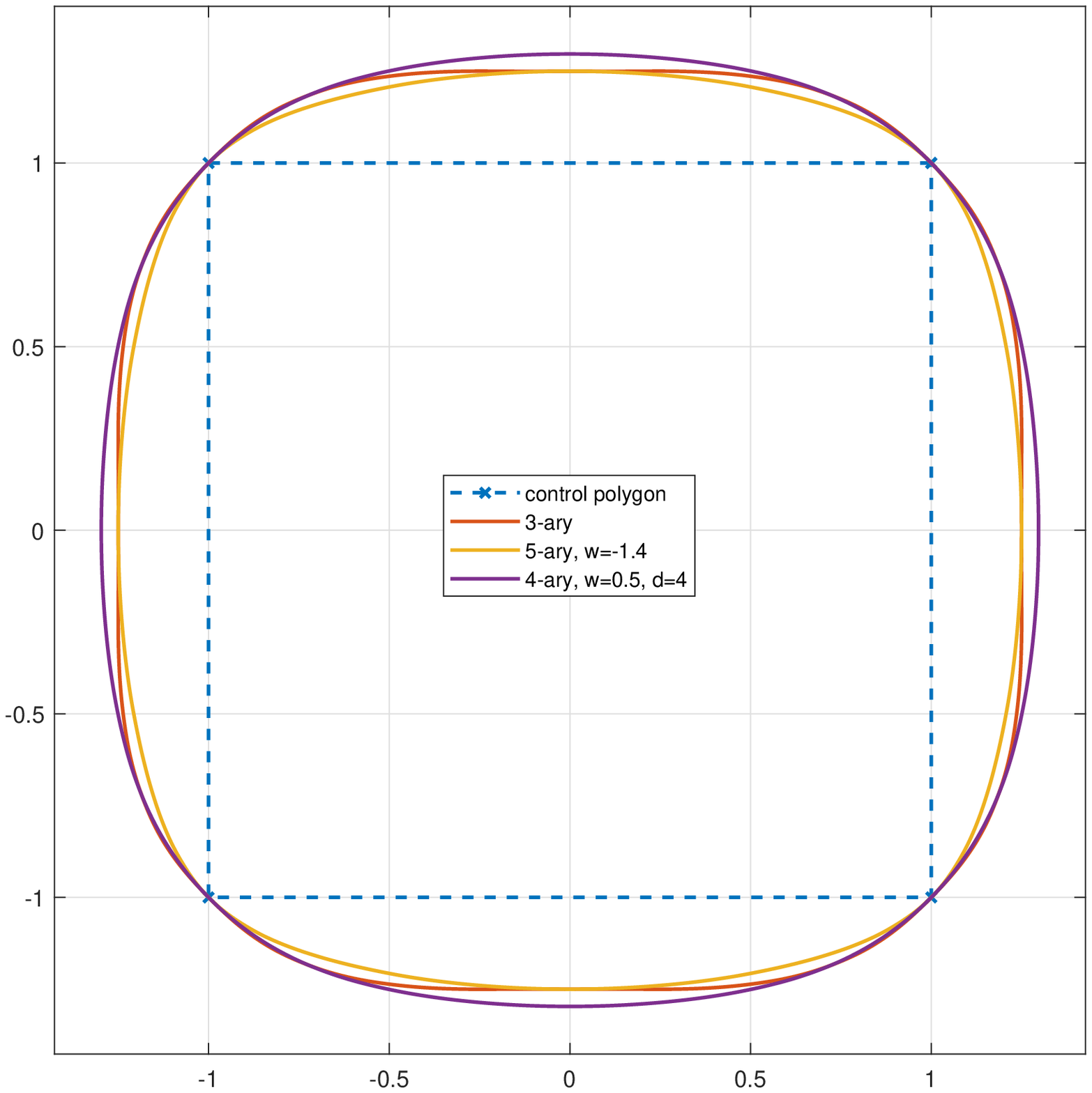}
	$ $\\ $ $\\
	\includegraphics[width=\textwidth]{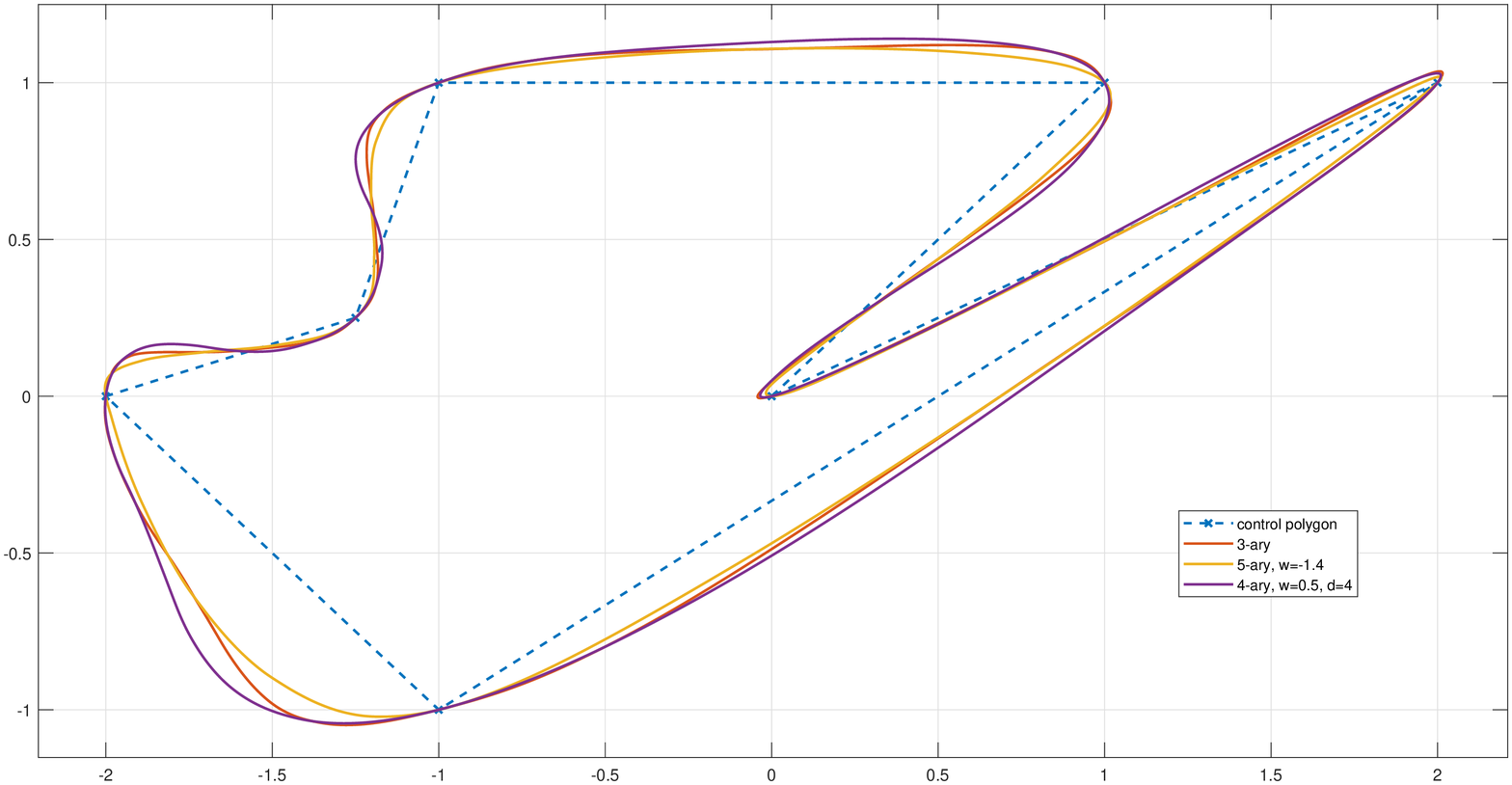}
	\caption{Comparison of the interpolations of the square and of a polygon with different angles given by the schemes related to the masks \eqref{eq:3ary_mask}, \eqref{eq:5ary_mask} with $w=-1.4$ and \eqref{eq:4ary_mask} with \eqref{eq:4ary_3rep} and $w=0.5$.}
\label{fig:confront}
\end{figure}

%

\section{Conclusions and Future Works} \label{sec:concl}

With this work, we open the door to an ignored room in the palace of subdivision, introducing a complete characterization of the symbols of dual univariate interpolatory schemes. From a theoretical point of view, this completes the well established theory of primal interpolatory schemes, while providing interesting tools for applications as well. From this point on, several directions of investigation start, such as the search for closed form of the symbols for subfamilies of dual interpolatory scheme (e.g. as the one existing for the Dubuc-Deslauriers primal interpolatory scheme), the analysis of the curvature, the study of the relationship with wavelet and frames and the design of bivariate dual interpolatory subdivision schemes.


\section*{Acknowledgements}
This research has been accomplished within the \emph{Research ITalian network on Approximation} (RITA).
The authors are members of the INdAM Research group GNCS, which has partially supported this work.


\bibliographystyle{siam}
\bibliography{dual_interpolatory_schemes_v3.bbl}
\end{document}